\documentclass[11pt]{amsart}

\usepackage{amsmath,amsthm,verbatim,amssymb,amsfonts,amscd, graphicx,color, framed,enumerate,stackrel}
\usepackage[left=1in,right=1in,top=1in,bottom=1in]{geometry}
\usepackage[colorlinks]{hyperref}
\usepackage{tikz}
\usepackage{pgfplots}
\pgfplotsset{width=10cm,compat=1.9}

\usepackage[most]{tcolorbox}
\usepackage{chemfig, chemnum}
\newtcolorbox[blend into=figures]{myfigure}[2][]{float=htb, title={#2},#1}

\usepackage[normalem]{ulem}

\def\R{\mathbb R}
\newcommand{\RR}{\mathbb{R}}
\def\Z{\mathbb Z}

\def\k{\kappa}

\def\deg{\text{deg}}

\def\Im{\text{Im}}

\def\dim{\text{dim}}

\newtheorem{theorem}{Theorem}[section]
\newtheorem{lemma}[theorem]{Lemma}

\newtheorem{proposition}[theorem]{Proposition}

\newtheorem{corollary}[theorem]{Corollary}

\theoremstyle{definition}

\newtheorem{definition}[theorem]{Definition}

\newtheorem{example}[theorem]{Example}
\newtheorem{question}[theorem]{Question}

\newtheorem{remark}[theorem]{Remark}
\newtheorem{notation}[theorem]{Notation}

\newtheoremstyle{example_contd}
{0.3cm}% Space above
{0.3cm}% Space below
{\upshape}% Body font
{}% Indent amount (empty = no indent, \parindent = para indent)
{\bfseries\scshape}% Thm head font
{.}% Punctuation after thm head
{0.5em}% Space after thm head (\newline = linebreak)
{ \thmname{#1} \thmnumber{ #2}\thmnote{#3} (continued)}% Thm head spec

\theoremstyle{example_contd}
\newtheorem*{example_contd}{\hspace*{-0.45em}Example}
\newcommand{\lradot}{%
  \mathrel{\ooalign{\hfil$\vcenter{
   \hbox{$\mkern3mu\scriptscriptstyle\bullet$}}$\hfil\cr$\longleftrightarrow$\cr}
  }%
}

\begin{document}

%-------
% TITLE
%-------

\title[Absolute concentration robustness and multistationarity]{Absolute concentration robustness and multistationarity in reaction networks: Conditions for coexistence}

\author{Nidhi Kaihnsa, Tung Nguyen, and Anne Shiu}

\maketitle

\begin{abstract}
Many reaction networks arising in applications are multistationary, that is, they have the capacity for more than one steady state; while some networks exhibit absolute concentration robustness (ACR), which means that some species concentration is the same at all steady states. 
Both multistationarity and ACR are significant in biological settings, 
but only recently has attention focused on the possibility for these properties to coexist. 
Our main result states that such coexistence 
in at-most-bimolecular networks (which encompass most networks arising in biology) 
 requires at least $3$ species, $5$ complexes, and $3$ reactions.  
We prove additional bounds on the number of reactions 
for general networks based on the number of linear conservation laws.
Finally, we prove that, outside of a few exceptional cases,  ACR is equivalent to non-multistationarity for bimolecular networks that are small (more precisely, one-dimensional or up to two species).  Our proofs involve analyses of systems of sparse polynomials, 
and we also use classical results from chemical reaction network theory.


\end{abstract}

% SET THIS FOR ARXIV
\date{\today}

%\tableofcontents

\section{Introduction}\label{sec:introduction}

A mass-action kinetics system exhibits {\em absolute concentration robustness} (ACR) if the steady-state value of at least one species is robust to fluctuations in initial concentrations of all species~\cite{ACR}. %,ACR_design}.
Another biologically significant property is 
the
existence of multiple steady states, that is, {\em multistationarity}. Significantly, this property has been linked to cellular decision-making and switch-like responses~\cite{Kholodenko2010, tyson-albert}.

As both ACR and multistationarity are important properties, it is perhaps surprising that their relationship was  explored only recently, when the present authors with Joshi showed that  
ACR and multistationarity together -- or even ACR by itself -- is highly atypical in randomly generated reaction networks. This result dovetails with the fact that the two properties are somewhat in opposition, as multiple steady states are not in general position in the presence of ACR.

The results of Joshi {\em et al.}
are asymptotic in nature 
(as the number of species goes to infinity), 
and they pertain to networks that are at-most-bimolecular (which is typical of networks arising in biology) and reversible (which is not)~\cite{joshi-kaihnsa-nguyen-shiu-1}.  
 This naturally leads to the following question: 
 
 \begin{question} \label{q:intro}
For multistationarity and ACR to coexist, how many species, reactions, and complexes are needed?
		 Which networks (without the requirement of being reversible) of small to medium size allow such coexistence? 
%	\end{enumerate}
\end{question}

Another motivation for Question~\ref{q:intro} comes from synthetic biology. In order to design reaction networks with certain dynamical properties, we need to better understand the design principles that allow for such behaviors, as well as the constraints on the size (such as the minimum numbers of species, reaction, and complexes) of such networks.

Our work focuses on answering Question~\ref{q:intro}.
Broadly speaking, our results fall into two categories:
(i) results that give lower bounds on the dimension of a network or its number of species, reactions, or complexes;  
and (ii) results for certain classes of networks (one-dimensional, up to $2$ species, and so on).
Our primary focus is on at-most-bimolecular networks, but we also present results on general networks.

In the first category, 
our results are summarized in the following theorem, which gives some minimum requirements for ACR and nondegenerate multistationarity to coexist. This coexistence is typically on a nonzero-measure subset of the parameter space of reaction rate constants.

\begin{theorem}[Main result] \label{thm:summary-theorem}
Let $G$ be an at-most-bimolecular reaction network with $n$ species 
such that there exists a vector of positive rate constants $\kappa^*$  
such that the mass-action system $(G,\kappa^*)$  has ACR and is nondegenerately multistationary. Then $G$ has:
	\begin{enumerate}
		\item at least $3$ species (that is, $n \geq 3$),
		\item at least $3$ reactant complexes (and hence, at least 3 reactions) and at least 5 complexes (reactant and product complexes), and
		\item dimension at least $2$.
	\end{enumerate}
	If, additionally, $G$ is full-dimensional (that is, $G$ has no linear conservation laws), then $G$ has:
	\begin{enumerate}
		\setcounter{enumi}{3}
		\item at least $n+2$ reactant complexes (and hence, at least $n+2$ reactions), and
		\item dimension at least $3$.
	\end{enumerate}

\end{theorem}

For the proof of Theorem~\ref{thm:summary-theorem}, we refer the reader to Section~\ref{sec:preliminary-results} for part~(3) (Lemma~\ref{lem:1-dim-no-mss}); 
Section~\ref{sec:preclude-mss-or-acr}
for parts~(1), (2), and~(5) (Theorem~\ref{thm:at-least-3-and-5}); and
Section~\ref{sec:results-general} for part~(4) (Theorem~\ref{theorem:smallestACRMSS}). 
Additionally, many of the lower bounds in Theorem~\ref{thm:summary-theorem} are tight. 
Indeed, this is shown for parts~(1)--(3) through the following network:
$\{A+B \to 2C \to 2B,~ C\to A \}$ (Example~\ref{example:min3}).  
As for part~(4), this bound is proven for networks that need not be at-most-bimolecular, and its tightness is shown in that context (Proposition~\ref{prop:family-min-num-reactants}).

While Theorem~\ref{thm:summary-theorem} concerns {\em nondegenerate} multistationarity, we also investigate the capacity for ACR together with {\em degenerate} multistationarity, specifically,  in networks with $4$ reactant complexes  (Proposition~\ref{prop:4rxns3species}).  
Finally, we prove two additional results in the spirit of Theorem~\ref{thm:summary-theorem}.  
The first states that $3$ is the minimum number of pairs of reversible reactions needed (in reversible networks) 
for multistationarity, even without ACR (Theorem~\ref{prop:2-rev-rxn-no-mss}).  
The second concerns networks that are not full-dimensional, and states the minimum number of reactant complexes needed for the coexistence of ACR and nondegenerate multistationarity
is $n-k+1$, where $1 \leq k \leq n-2$ is the number of linearly independent conservation laws
(Theorem~\ref{thm:min-num-reac-with-k-cons-law}).

As for our second category of results, we start with one-dimensional networks, 
a class of networks for which ACR~\cite{acr-dim-1,MST}, multistationarity~\cite{joshi2017small, shiu2019nondegenerate}, and even multistability \cite{tang-zhang} is well studied. 
Such networks do not allow for the coexistence of ACR and nondegenerate multistationarity  (Proposition~\ref{prop:1-d-no-coexistence}).  Moreover, one-dimensional bimolecular networks can only be multistationary if they are degenerately so (Lemma \ref{lem:1-dim-no-mss}).  Moreover, we explicitly characterize all such degenerate networks (Lemma~\ref{lem:1-d-infinite-steady-states}). 
Here our proofs make use of recent results of Lin, Tang, and Zhang~\cite{lin-tang-zhang, tang-zhang}.

Another class of at-most-bimolecular networks we analyze are those with exactly 2 species (Section~\ref{sec:BinaryRN}). 
For such networks that are reversible, we characterize the property of
unconditional ACR, which means that ACR occurs for all possible values of rate constants
(Theorem~\ref{thm:summary-2-species}). 
As for networks that need not be reversible, we show that ACR and multistationarity can coexist, but only in a degenerate way.  
Moreover, up to relabelling species, only two such networks allow such coexistence for a nonzero-measure subset of the space of reaction rate constants (Theorem~\ref{thm:2species}). 

Our works fits into a growing body of literature that explores the minimal conditions needed for various dynamical behaviors,
 including the two properties that are the focus of the current work:
multistationarity~\cite{joshi2017small, 
lin-tang-zhang, 
shiu2019nondegenerate}
and ACR~\cite{acr-dim-1,MST}.  There are additional such studies on 
multistability~\cite{tang-wang-hopf}
and 
Hopf bifurcations~\cite{banaji-boros, banaji-boros-hofbauer-3-rxn,MR4241183,tang-zhang,smallestHopf} (which generate periodic orbits).  For instance, in analogy to Theorem~\ref{thm:summary-theorem} above, the presence of Hopf bifurcations requires an at-most-bimolecular network to have at least $3$ species, $4$ reactions, and dimension $3$~\cite{banaji-boros,smallestHopf}.

This article is organized as follows: Section~\ref{sec:preliminaries} introduces reaction networks, multistationarity, and ACR. Section~\ref{sec:preliminary-results} contains several results on steady states and their nondegeneracy.  We use these results in Sections~\ref{sec:preclude-mss-or-acr} and~\ref{sec:results-general} to prove our main results.  
We conclude with a discussion in Section~\ref{sec:discussion}.

\section{Background}\label{sec:preliminaries}

This section recalls the basic setup and definitions involving reaction networks (Section~\ref{sec:rxn-network}), 
the dynamical systems they generate (Section~\ref{sec:mass-action}), absolute concentration robustness (Section~\ref{sec:ACR-system}),  and a concept pertaining to networks with only $1$ species: ``arrow diagrams'' (Section~\ref{sec:arrow-diagrams}).

\subsection{Reaction networks} \label{sec:rxn-network}
  
A {\em reaction network} $G$ is a directed graph in which the vertices are non-negative-integer linear combinations of {\em species} $X_1, X_2, \ldots, X_n$.  Each vertex is a {\em complex}, and we denote the complex at vertex $i$ by $y_i=\sum_{j=1}^{n}y_{ij}X_j$ (where $y_{ij}\in\Z_{\geq 0}$) or $y_i=(y_{i1}, y_{i2}, \dots, y_{in} )$.
Throughout, we assume that each species $X_i$, where $i=1,2,\dots,n$, appears in at least one complex.

Edges of a network $G$ are {\em reactions}, and it is standard to represent a reaction $(y_i, y_j)$ by $y_i \to y_j$.  
In such a reaction,  $y_i$ is the \textit{reactant complex}, and $y_j$ is the \textit{product complex}. 
A species $X_k$ is a {\em catalyst-only species in reaction} $y_i \to y_j$ if $y_{ik} = y_{jk}$. 
In examples, it is often convenient to write species as $A,B,C,\dots$ (rather than $X_1,X_2,X_3,\dots$) and also to view a network as a set of reactions, where the sets of species and complexes are implied. 

\begin{example} \label{ex:generalized-degenerate-network}
	The reaction network $\{ 0 \leftarrow A \to 2A ,~ B \leftarrow A+B\}$ has 2 species, 5 complexes, and 3 reactions. The species $B$ is a catalyst-only species in the reaction $B \leftarrow A+B$.
\end{example}

A reaction network is \textit{reversible} if every edge of the graph is bidirected.  
A reaction network is \textit{weakly reversible} if every connected component of the graph is strongly connected.   
Every reversible network is weakly reversible.

\begin{example} \label{ex:reversible-mss-network}
	The following network is reversible: 
	$\{
	A+B {\leftrightarrows} 2A, ~2B {\leftrightarrows} A,  0 \leftrightarrows B \}$.
\end{example}

One focus of our work is on \textit{at-most-bimolecular} reaction networks (or, for short, {\em bimolecular}), which means that every complex $y_i$ satisfies 
$y_{i1}+y_{i2} + \dots + y_{in} \leq 2$.
%$\sum_j^n y_{ij}\leq 2$. 
Equivalently, each complex has the form $0$, $X_i$, $X_i+X_j$, or $2X_i$ (where $X_i$ and $X_j$ are species). The networks in  Examples~\ref{ex:generalized-degenerate-network}--\ref{ex:reversible-mss-network} are bimolecular.

\subsection{Mass-action systems} \label{sec:mass-action}
Let $r$ denote the number of reactions of $G$. 
We write the $i$-th reaction as $y_i \to y_i'$ 
and assign to it a positive 
\textit{rate constant} 
$\kappa_{i}\in \RR_{> 0}$. 
The \textit{mass-action system} arising from a network $G$ and a vector of positive rate constants $\kappa=(\kappa_1, \kappa_2, \dots, \kappa_r)$, which we denote by $(G,\kappa)$, is the following dynamical system arising from mass-action kinetics: 
    \begin{equation}\label{eq:mass_action_ODE}
    \frac{dx}{dt} 
    ~=~
    \sum_{i=1}^r  \kappa_{i} x^{y_i} (y_i'-y_i) ~=:~ f_{\kappa}(x)~,
    \end{equation}
where $x^{y_i} := \prod_{j=1}^n x_j^{y_{ij}}$.
Observe that the right-hand side of the ODEs~\eqref{eq:mass_action_ODE} consists of polynomials 
$f_{\kappa,i}$, 
for $i=1,\dots,n$. For simplicity, we often write $f_i$ instead of $f_{\kappa,i}$.

The question of which polynomials $f_i$ can appear as right-hand side of mass-action ODEs is answered in the following result~\cite[Theorem 3.2]{HT}.

\begin{lemma} \label{lem:hungarian}
	Let $f:\mathbb{R}^n \rightarrow \mathbb{R}^n$ be a polynomial function, that is, assume that $f_i \in \mathbb{R}[x_1, x_2 \dots, x_n]$ for $i=1,2,\dots, n$.  Then $f$ arises as the right-hand side of the differential equations~\eqref{eq:mass_action_ODE} (for some choice of network $G$ and vector of positive rate constants $\kappa$) if and only if, for all $i=1,2,\dots,n$, every monomial in $f_i$ with negative coefficient is divisible by $x_i$.
\end{lemma}

Next, observe that the mass-action ODEs~\eqref{eq:mass_action_ODE} are in the linear subspace of $\mathbb{R}^n$ spanned by all reaction vectors $y_i' - y_i$ (for $i=1,2,\dots, r$).  We call this the \textit{stoichiometric subspace} and denote it by $S$.  
The {\em dimension} of a network is the dimension of its stoichiometric subspace.  
(This dimension is sometimes called the ``rank''~\cite{splitting-banaji, tang-wang-hopf}.)
In particular, if $\dim(S)=n$ (that is,
$S= \mathbb{R}^n$), we say that $G$ is {\em full-dimensional}.  

A trajectory $x(t)$ of~\eqref{eq:mass_action_ODE} with initial condition $x(0) = x^0 \in \R_{>0}^n$ remains, for all positive time, in the following \textit{stoichiometric compatibility class} of $G$~\cite{FeinbergLec79}:
\begin{align} \label{eq:scc}
P_{x(0)} ~:=~ (x(0)+S)\cap \RR_{\geq 0}^n~.
\end{align}

For full-dimensional networks, there is a unique stoichiometric compatibility class: $P=\RR_{\geq 0}^n$.  For networks that are not full-dimensional, every nonzero vector $w$ in $S^{\perp}$ yields a (linear) {\em conservation law} $\langle w,x \rangle = \langle w, x(0) \rangle$ that is satisfied by every $x \in P_{x(0)}$, where $\langle -,- \rangle$ denotes the usual inner product on $\R^n$.

\begin{example_contd}[\ref{ex:generalized-degenerate-network}] \label{ex:generalized-degenerate-network-ode} 
	The network $\{ 0 \overset{\kappa_1}{\leftarrow} A \overset{\kappa_2}{\to} 2A ,~ B \overset{\kappa_3}{\leftarrow} A+B\}$ 
	has a one-dimensional stoichiometric subspace (spanned by $(1,0)$) and generates the following mass-action ODEs~\eqref{eq:mass_action_ODE}:
	\begin{align} \label{ex:ode-gen-deg-net}
	&\frac{dx_1}{dt} ~=~ -\kappa_1x_1+\kappa_2x_1-\kappa_3x_1 x_2=x_1(-\kappa_1+\kappa_2-\kappa_3 x_2)\\
	&\frac{dx_2}{dt} ~=~ 0.
	\notag
	\end{align}
	%	
	%deficiency of the network is $\delta=5-2-1=2$
	Observe that the negative monomials in the first ODE are $-\kappa_1x_1$ and $-\kappa_3x_1 x_2$, and each of these is divisible by $x_1$, which is consistent with Lemma~\ref{lem:hungarian}. 
	Next, the stoichiometric compatibility classes~\eqref{eq:scc} are rays of the following form (where $T>0$):
	\begin{align} \label{eq:SCC-for-gen-deg-net}
	\{ (x_1,x_2) \in \mathbb{R}^2_{\geq 0} \mid x_2 = T\}~.
	\end{align}
The equation $ x_2 = T$ is the unique (up to scaling) conservation law.
\end{example_contd}

A {\em steady state} of a mass-action system is a non-negative vector $x^*\in \R_{\geq 0}^n$ at which the right-hand side of the ODEs \eqref{eq:mass_action_ODE} vanishes: $f_{\kappa}(x^*)=0$. 
Our main interest in this work is in {\em positive} steady states $x^*\in \R_{> 0}^n$.
The set of all positive steady states of a mass-action system 
can have positive dimension in $\RR^n$, 
but this set typically intersects each stoichiometric compatibility class in finitely many points~\cite{feliu-henriksson-pascual}.  
Finally, a steady state $x^*$ is \textit{nondegenerate} if $\Im(df_{\kappa}(x^*)|_S)=S$, where $df_{\kappa}(x^*)$ is the Jacobian matrix of $f_\kappa$ evaluated at $x^*$.

We consider multiple steady states at two levels: systems and networks.  
A mass-action system $(G,\kappa)$ is \textit{multistationary} (respectively, {\em nondegenerately multistationary}) 
if there exists a stoichiometric compatibility class having more than one positive steady state (respectively, nondegenerate positive steady state). 
A reaction network $G$ is \textit{multistationary} if there exists a vector of positive rate constants $\kappa$ such that $(G,\kappa)$ is multistationary.  
For a reaction network $G$, we let ${\rm cap}_{pos}(G)$ 
(respectively,  ${\rm cap}_{nondeg}(G)$)
denote the maximum possible number of positive steady states (respectively, nondegenerate positive steady states) in a stoichiometric compatibility class. 

%-------------------------------------
% EXAMPLE - CONTINUED
%-------------------------------------
\begin{example_contd}[\ref{ex:generalized-degenerate-network}] \label{ex:generalized-degenerate-network-steady-state}
	We return to the network $G=\{ 0 \overset{\kappa_1}{\leftarrow} A \overset{\kappa_2}{\to} 2A ,~ B \overset{\kappa_3}{\leftarrow} A+B\}$ 
	and its ODEs~\eqref{ex:ode-gen-deg-net}.  A direct computation reveals that when $\kappa_1 \geq \kappa_2$, there is no positive steady state.
	On the other hand, when $\kappa_2 > \kappa_1$, the steady states form exactly one stoichiometric compatibility class~\eqref{eq:SCC-for-gen-deg-net} -- namely, the one given by $T = (\kappa_2 - \kappa_1)/\kappa_3$ -- and all such steady states are degenerate.  Hence, $G$ is multistationary but {\em not} nondegenerately multistationary.
\end{example_contd}

%-------------------------------------
% EXAMPLE - continued
%-------------------------------------
\begin{example_contd}[\ref{ex:reversible-mss-network}] \label{ex:3-revesible-reactions}
	The following (full-dimensional) reaction network and indicated rate constants yield a mass-action system with $3$ nondegenerate positive steady states~\cite[Remark 3.6]{joshi-kaihnsa-nguyen-shiu-1}:
	\begin{align} \label{eq:3-rev-rxn-rates}
	&
	\left\{
	A+B \stackrel[1/4]{1/32}{\leftrightarrows} 2A, \quad 2B \stackrel[1/4]{1}{\leftrightarrows} A, \quad 0 \stackrel[1]{1}{\leftrightarrows} B \right\} ~.
	\end{align}
	Therefore, this network is nondegenerately multistationary.
\end{example_contd}

\subsection{Deficiency and absolute concentration robustness} \label{sec:ACR-system}

The {\em deficiency} of a reaction network~$G$ is 
$\delta = m - \ell- \dim(S)$, where $m$ is the number of vertices (or complexes), $\ell$ is the number of connected components of $G$ (also called {\em linkage classes}), and $S$ is the stoichiometric subspace. 
The deficiency is always non-negative~\cite{FeinbergLec79}, 
and it plays a central role in many classical results on the dynamical properties of mass-action systems~\cite{AC:non-mass,ACK:product,AN:non-mass,F1,H,H-J1}.

Two such results are stated below. These results, which are due to Feinberg and Horn~\cite{Feinberg1987,FeinDefZeroOne, H}, are stated for weakly reversible networks (the setting in which we use these results later).

%----------------------------------
% LEMMA: Deficiency zero theorem
%----------------------------------
\begin{lemma}[Deficiency-zero theorem] \label{lem:thm:def-0}
	Deficiency-zero networks are not multistationary.  Moreover, if 
	$G$ is a weakly reversible network with deficiency zero, then for every vector of positive rate constants $\kappa$, 
	the mass-action system $(G,\kappa)$ admits a unique positive steady state in every stoichiometric compatibility class.
\end{lemma} 

%----------------------------------
% LEMMA: Deficiency zero theorem
%----------------------------------
\begin{lemma}[Deficiency-one theorem] \label{lem:thm:def-1}
	Consider a weakly reversible network $G$ with connected components (linkage classes) $G_1$, $G_2$, \dots, $G_{\ell}$.  
	Let $\delta$ denote the deficiency of $G$, and (for all $i=1,2,\dots, \ell$) let $\delta_i$ denote the deficiency of $G_i$.  
	Assume the following:
	\begin{enumerate}
		\item $\delta_i \leq 1$ for all $i=1,2,\dots,\ell$, and
		\item $\delta_1+\delta_2+ \dots + \delta_{\ell} = \delta$.
	\end{enumerate}
	Then $G$ is not multistationary: 
	for every vector of positive rate constants $\kappa$, 
	the mass-action system $(G,\kappa)$ admits a unique positive steady state in every stoichiometric compatibility class.
\end{lemma}

Our next topic, ACR, like multistationarity, is analyzed at the level of systems and also networks.  

\begin{definition}[ACR] \label{def:acr}
	Let $X_i$ be a species of a reaction network $G$ with $r$ reactions.
	\begin{enumerate}
		\item For a fixed vector of positive rate constants $\kappa \in \mathbb{R}^r_{>0}$, 
		the mass-action system $(G,\kappa)$ has {\em absolute concentration robustness} (ACR) in $X_i$ if $(G,\kappa)$ has a positive steady state and in every positive steady state $x \in \RR_{> 0}^n$ of the system, the value of $x_i$  is the same. This value of $x_i$ is the \textit{ACR-value} of $X_i$. 
		\item The reaction network $G$ 
		has \textit{unconditional ACR}  in species $X_i$ if,  for every vector of positive rate constants $\kappa \in \mathbb{R}^r_{>0}$, the mass-action system $(G,\kappa)$ has ACR in $X_i$. 
	\end{enumerate}
\end{definition}

\begin{remark}[Existence of positive steady states] \label{rem:reversible}
	ACR requires the existence of a positive steady state (Definition~\ref{def:acr}(1)).  
	This requirement is sometimes not included in definitions of ACR in the literature.  
	However, this is not an extra requirement for some of the networks we consider, namely, 
	weakly reversible networks, for which 
	positive steady states are guaranteed to exist (see Deng {\em et al.}~\cite{Deng} and Boros~\cite{boros2019existence}).
\end{remark}

\begin{remark}\label{rem:measure}
	The property of unconditional ACR is often too restrictive. 
	Thus, many of our results focus on ACR (or other properties) that hold for some full-dimensional subset of the parameter space of rate constants $\R^r_{> 0}$ (where $r$ is the number of reactions of a given network).
	The Lesbesgue measure of such a subset is nonzero. For simplicity, we use ``measure'' to mean Lebesgue measure. 
\end{remark}

%-------------------------------------
% EXAMPLE - CONTINUED
%-------------------------------------
\begin{example_contd}[\ref{ex:generalized-degenerate-network}] \label{ex:generalized-degenerate-network-AR}
	We revisit
	the network $\{ 0 \overset{\kappa_1}{\leftarrow} A \overset{\kappa_2}{\to} 2A ,~ B \overset{\kappa_3}{\leftarrow} A+B\}$. 
	From our earlier analysis, the mass-action system has ACR in $B$ when $\kappa_2 > \kappa_1$ (which defines a nonzero-measure subset of the rate-constants space $\R^3_{> 0}$), but lacks ACR when 
	$\kappa_2 \leq \kappa_1$ (as there are no positive steady states).
\end{example_contd}

\begin{example} \label{ex:why-need-reversible-in-prop}
	Consider the following network $G$, which is bimolecular and full-dimensional:
	\begin{align*}
	\{
	2X_2 \xleftarrow{\kappa_3} X_2 \stackrel[\kappa_2]{\kappa_1}{\rightleftarrows} X_1+X_2  \xrightarrow{\kappa_4} X_1 \}. 
	\end{align*}
	The mass-action ODEs are as follows:
	\begin{align}
	\dot{x}_1
	&~=~
	\kappa_1x_2 -\kappa_2x_1 x_2
	~=~
	(\kappa_1-\kappa_2x_1)x_2 \notag \\ 
	\dot{x}_2
	&~=~
	\kappa_3x_2-\kappa_4x_1x_2
	~=~
	(\kappa_3-\kappa_4x_1)x_2~.
	\end{align}
	When
	$\tfrac{\kappa_1}{\kappa_2} \neq \tfrac{\kappa_3}{\kappa_4}$, there are no positive steady states and hence no ACR.  
	Now assume $\tfrac{\kappa_1}{\kappa_2}=\tfrac{\kappa_3}{\kappa_4}$.  
	In this case, 
	the positive steady states are defined by the line $x_1=\tfrac{\kappa_1}{\kappa_2}$, and so the system is multistationary and has ACR in species $X_1$. However, all the steady states of this system are degenerate. 
\end{example}

\begin{example} \label{ex:unconditional-ACR}
	Consider the following network~\cite[Example 2.6]{joshi-kaihnsa-nguyen-shiu-1}, which we call $G$:
	\begin{align*}
	\left\{
	A \stackrel[\kappa_1]{\kappa_2}{\leftrightarrows} A + B, 
	\quad 
	2B \stackrel[\kappa_3]{\kappa_4}{\leftrightarrows} 3B, 
	\quad A 
	\stackrel[\kappa_5]{\kappa_6}{\leftrightarrows} 2A 
	\right\}~.
	\end{align*}  
	The mass-action ODEs~\eqref{eq:mass_action_ODE} are as follows:
	\begin{align*}
	&\frac{dx_1}{dt} ~=~ \kappa_5x_1 - \kappa_6x_1^2\\
	&\frac{dx_2}{dt}~ =~ \kappa_1x_1-\kappa_2x_1x_2+\kappa_3x_2^2-\kappa_4x_2^3.
	\end{align*}
	It follows that $G$ has unconditional ACR in species $A$ with ACR-value $\kappa_5/\kappa_6$ (the existence of positive steady states comes from the fact that $G$ is reversible; recall Remark~\ref{rem:reversible}).
\end{example}

The following result, which is~\cite[Lemma 5.1]{MST}, concerns ACR in one-dimensional networks.

\begin{lemma} \label{lem:1-dim-reactants}
	Let $G$ be a one-dimensional network with species $X_1, X_2, \dots, X_n$.
	If $G$ has unconditional ACR in some species $X_{i^*}$, then 
	the reactant complexes of $G$ differ only in species $X_{i^*}$ 
	(more precisely, if $y$ and $\widetilde y$ are both reactant complexes of $G$, then $y_i=\widetilde{y}_i$ for all 
	$i \in \{1,2,\dots,n\} \smallsetminus \{i^* \} $).
\end{lemma}

\subsection{Arrow diagrams} \label{sec:arrow-diagrams}

In this subsection, we recall the \textit{arrow diagrams} associated to one-species networks. These diagrams are useful for stating results about such 
networks~\cite{joshi2017small, MST,mv-small-networks}.

\begin{definition}[Arrow diagram] \label{def:arrow-diagram}
	Let $G$ be a reaction network with only one species $X_1$.  Let $m$ denote the number of (distinct) reactant complexes of $G$,  which we list in increasing order of molecularity: $a_1 X_1, a_2X_1, \dots, a_m X_1$ (so, $a_1 < a_2 < \dots < a_m$).  The {\em arrow diagram} of $G$ is the vector  
	$\rho = (\rho_1, \rho_2, \ldots , \rho_m) \in  \{\to , \leftarrow, \lradot \}^m$ defined by: 
	\begin{equation*}
	\rho_i	
	~:=~ 
	\left\lbrace\begin{array}{ll}
	\to & \text{if for every reaction $a_i X_1 \to bX_1$ in $G$, the inequality $b > a_i$ holds} \\
	\leftarrow & \text{if for every reaction $a_i X_1 \to bX_1$ in $G$, the inequality $b < a_i$ holds} \\
	\lradot & \text{otherwise.}
	\end{array}\right.
	\end{equation*}
\end{definition}

\begin{example}  ~
	\begin{enumerate}[(1)]
		\item The network $\{0 \leftarrow A,~ 2A \to 3A\}$ has arrow diagram $(\leftarrow, \to)$.
		\item The network $\{0 \leftarrow A,~ A \to 2A,~ 2A \to 3A\}$ has arrow diagram $(\lradot, \to)$.
	\end{enumerate}
\end{example}

It is often useful to consider the arrow diagrams of ``embedded'' one-species networks, as follows.

\begin{definition} \label{def:embeddednetwork}
	Let $G$ be a reaction network with species $X_1, X_2, \dots, X_n$.  Given a species $X_i$, the corresponding {\em embedded one-species network} of $G$ is obtained by deleting some (possibly empty) subset of the reactions, replacing each remaining reaction $a_1 X_1 + a_2 X_2 + \dots +a_s X_s \to b_1 X_1 + b_2 X_2 + \dots +b_s X_s $ by the reaction $a_i X_i \to b_i X_i$, and then deleting any trivial reactions (i.e, reactions of the form $a_i X_i \to a_i X_i$, in which the reactant and product complexes are equal) and keeping only one copy of duplicate reactions.
\end{definition}

\begin{example} \label{ex:embed}
	Consider the network $G=\{ 0\leftrightarrows  B \to A \}$.  The following networks are embedded one-species networks of $G$:
	$\{0 \to B\}$, $\{0 \leftarrow B \}$, $\{0 \leftrightarrows B \}$, and $\{0 \to A\}$.
\end{example}

\section{Results on steady states and nondegeneracy} \label{sec:preliminary-results}

This section contains results on the steady states of mass-action systems.  We use these results in later sections 
to prove our main results. 
Section~\ref{sec:ODE-reactants} analyzes the steady states of full-dimensional networks, 
while Section~\ref{sec:conservation-laws}
pertains to non-full-dimensional networks.
Next, Section~\ref{sec:bimol-prelim-results} focuses on bimolecular networks and investigates scenarios in which the right-hand side of a mass-action ODE vanishes.
Finally, 
Section~\ref{sec:multismallnetworks} concerns bimolecular networks that are reversible.
%, $3$ pairs of reversible reactions are necessary for multistationarity.

\subsection{Full-dimensional networks }\label{sec:ODE-reactants}
Consider a reaction network $G$ with 
 $n$ species, $r$ reactions, and exactly $j$ reactant complexes\footnote{{ A network has {\em exactly $j$ reactant complexes} if the set of distinct reactant complexes has size $j$.}}; and let $\kappa^* \in \mathbb{R}^r_{>0}$ be a vector of positive rate constants.
We often rewrite the mass-action ODE system~\eqref{eq:mass_action_ODE} for $(G,\kappa^*)$
 as follows:
\begin{align}\label{eqn:ODEexpression}
\begin{bmatrix} 
dx_1/dt \\
dx_2/dt \\
\vdots \\
dx_n/dt \\
\end{bmatrix}
~=~
N 
\begin{bmatrix} 
m_1 \\
m_2 \\
\vdots \\
m_j 
\end{bmatrix}	~,
\end{align} 
where $N$ is an $(n \times j)$-matrix (with real entries) and $m_1,m_2,\dots,m_j$ are distinct monic monomials in $x_1, x_2, \ldots,x_n$ given by the reactant complexes.

\begin{example_contd}[\ref{ex:why-need-reversible-in-prop}]
The network 
$	\{
	2X_2 \xleftarrow{\kappa_3} X_2 \stackrel[\kappa_2]{\kappa_1}{\rightleftarrows} X_1+X_2  \xrightarrow{\kappa_4} X_1 \} 
$
	has two reactant complexes, which yield the monomials $m_1:=x_2$ and $m_2:=x_1x_2$.  
Consider $(\kappa_1,\kappa_2,\kappa_3,\kappa_4)=(1,2,3,6)$ (so, $\tfrac{\kappa_1}{\kappa_2}=\tfrac{\kappa_3}{\kappa_4}$ holds). Now the matrix $N$, as in~\eqref{eqn:ODEexpression}, is as follows:
	\begin{align*}
		N ~:=~
		\begin{bmatrix} 
		1&-2 \\
		3&-6 
		\end{bmatrix}~.
	\end{align*}
This matrix $N$ does not have full rank, and we saw earlier that all steady states of this mass-action system are degenerate.  In the next result, part~(1) asserts that this phenomenon holds in general.
\end{example_contd}

%The results in this subsection pertain to full-dimensional networks.

\begin{proposition}[Nondegenerate steady states and the matrix $N$] \label{prop:rank-of-N}
	Let $G$ be a full-dimensional reaction network with $n$ species, and $\kappa^*$ be a vector of positive rate constants.
	Let $N$ be a matrix defined, as in~\eqref{eqn:ODEexpression}, by the mass-action ODE system of $(G,\kappa^*)$.
	\begin{enumerate}
		\item If $\operatorname{rank}(N) \leq n-1 $, then every positive steady state of $(G,\kappa^*)$ is degenerate.  
		\item If $\operatorname{rank}(N) = n$ and $G$ has exactly $n+1$ reactant complexes, then the positive steady states of $(G,\kappa^*)$ 
		are the positive roots of a system of binomial equations (sharing some common monomial $m_0$) of the following form:
			\[
			m_i - \beta_i m_{n+1} ~=~ 0 \quad \mathrm{for}~i=1,2,\dots, n~,
			\]
		where $\beta_1,\beta_2, \dots, \beta_{n} \in \mathbb{R}$ and $m_1, \dots , m_{n+1} $  are distinct monic monomials in $x_1, x_2, \ldots,x_n$.  
		\item If $G$ has exactly $n+1$ reactant complexes and $(G,\kappa^*)$ has a  nondegenerate, positive steady state, then $(G,\kappa^*)$ is \uline{not} multistationary. 
	\end{enumerate}
\end{proposition}

\begin{proof}
Assume $(G,\kappa^*)$ is a full-dimensional mass-action system in $n$ species, and let $N$ be as in~\eqref{eqn:ODEexpression}.

First, we prove~(1).  Assume $\operatorname{rank}(N) \leq n-1 $, and let $x^*$ be a positive steady state.   
It follows that the polynomials $f_i$, as in~\eqref{eq:mass_action_ODE}, are linearly dependent (over $\mathbb R$). Hence,
	 the Jacobian matrix -- even before evaluating at $x^*$ -- has rank less than $n.$ Thus, the image of the Jacobian matrix, after evaluating at $x^*$, has dimension less than $n$,  i.e, $\Im(df(x^*)|_S)\neq \mathbb{R}^n =S$.
	 Hence, $x^*$ is degenerate. 

Next, we prove~(2).
	As in equation~\eqref{eqn:ODEexpression}, we write the mass-action ODEs for $(G,\kappa^*)$ as 
	\begin{align*}
	\begin{bmatrix} 
	dx_1/dt \\
	dx_2/dt \\
	\vdots \\
	dx_n/dt \\
	\end{bmatrix}
	~=~
	N 
	\begin{bmatrix} 
	m_1 \\
	\vdots \\
	m_n \\
	m_{n+1}\\
	\end{bmatrix}	~,
	\end{align*} 
	where $N$ is $n \times (n+1)$ and the $m_i$'s are distinct monic monomials in $x_1,x_2, \ldots,x_n$.

	As $G$ is full-dimensional and $\operatorname{rank}(N) =n $, we can relabel the $m_i$'s, if needed, so that the square sub-matrix of $N$ formed by the first $n$ columns has rank $n$.  Thus, by row-reducing $N$, we obtain a matrix of the following form (where $\beta_1, \beta_2, \ldots,\beta_n \in \mathbb{R}$):  
	\begin{align*}
	N' ~:=~
	\left[\begin{array}{ccc|c} 
	&  &  & -\beta_1\\ 
	& {I_n} &  & -\beta_2 \\ 	
	&  &  & \vdots\\ 
	&  &  & -\beta_n\\ 
	\end{array}\right]~.
	\end{align*}
	
	We conclude from the above discussion that the positive steady states of $(G,\kappa^*)$ are the positive roots of the following $n$ binomial equations (which are in the desired form): 
	\begin{align} \label{eq:binomials-in-proof}
	m_i - \beta_i m_{n+1} ~=~0 \quad \quad {\rm for~}i=1,2,\ldots, n~.
	\end{align}

Before moving on to part (3), we summarize what we know (so we can use it later).
The positive steady states are the roots of the binomials~\eqref{eq:binomials-in-proof}, 
which we rewrite using Laurent monomials (our interest is in positive roots, so there is no issue of dividing by zero):
	\begin{align} \label{eq:laurent-binomials}
	x_1^{a_{i1}} 
		x_2^{a_{i2}} \dots
			x_n^{a_{in}}
			~:=~ \frac{m_i}{m_{n+1}} ~=~ \beta_i 		 \quad \quad {\rm for~}i=1,2,\ldots, n~.	
	\end{align}
We apply the natural $\log$ to~\eqref{eq:laurent-binomials} and obtain the following, which involves the $n \times n$ matrix $A:=(a_{ij})$:
\begin{align} \label{eq:linear-system-proof}
A 
\begin{pmatrix} \ln(x_1) \\ \ln(x_2) \\ \vdots \\ \ln(x_n) \end{pmatrix}
~=~
	\begin{pmatrix} \ln(\beta_1) \\ \ln(\beta_2) \\ \vdots \\ \ln(\beta_n)
	\end{pmatrix}
	~=:~ \ln (\beta)~. 
\end{align}

Now we prove~(3).  Assume $x^*$ is a nondegenerate, positive steady state.  (We must show that no other positive steady states exist.)  By part~(1), the $n \times (n+1)$ matrix $N$ has rank~$n$, so the proof of part (2) above applies.
Assume for contradiction that $x^{**}$ is a positive steady state, with $x^{**} \neq x^*$.  Then, by~\eqref{eq:linear-system-proof},
the linear system $Ay=\ln(\beta)$ has more than one solution, and so ${\rm rank}(A) \leq n-1$.  It follows that the set of positive steady states, 
$\{(e^{y_1}, e^{y_2}, \dots, e^{y_n}) \mid Ay=\ln(\beta)\}$, is positive-dimensional and so (by the Inverse Function Theorem and the fact that $G$ is full-dimensional) all positive steady states of $(G,\kappa^*)$ are degenerate.  This is a contradiction, as $x^*$ is nondegenerate.
\end{proof}

\begin{remark} \label{rem:toric}
	For algebraically inclined readers, observe that the equations in Proposition~\ref{prop:rank-of-N}(2) define a toric variety. 
	Additionally, every such variety has at most one irreducible component that intersects the positive orthant~\cite[Proposition 5.2]{CIK}. 
	This fact can be used to give a more direct proof of Proposition~\ref{prop:rank-of-N}(3).  
\end{remark}

\begin{remark} \label{rem:FHP}
The end of the proof of
Proposition~\ref{prop:rank-of-N}
concerns nondegenerate positive steady states 
and their relation to the dimension of the set of positive steady states.  More ideas in this direction are explored in the recent work of Feliu, Henriksson, and Pascual-Escudero~\cite{feliu-henriksson-pascual}.
\end{remark}

\begin{corollary}[When $f_i$ is zero]\label{cor:degeneracy}
	Let $G$ be a full-dimensional reaction network with $n$ species, let $\kappa^*$ be a vector of positive rate constants,
	and let $f_1,f_2,\dots, f_n$ denote the right-hand sides of the mass-action ODEs of $(G,\kappa^*)$.
	If $f_i$ is the zero polynomial, for some $i\in\{1,\ldots,n\}$, then 
	every positive steady state of $(G,\kappa^*)$ is degenerate.
\end{corollary}
\begin{proof}
This result follows directly from Proposition~\ref{prop:rank-of-N}(1) and the fact that, in this case, the rank of $N$, as in \eqref{eqn:ODEexpression}, is strictly less than $n$.
\end{proof}

The next two results pertain to networks with few reactant complexes (at most $n$, where $n$ is the number of species) 
and many reactant complexes (at least $n$), respectively.

\begin{proposition}[Networks with few reactants] \label{prop:2-or-3-rxns3species}
	Let $G$ be a reaction network with $n$ species. 
	\begin{enumerate}
		\item 
		If $G$ has exactly 1 reactant complex, then, for every vector of positive rate constants $\kappa^*$, the mass-action system $(G, \kappa^*)$ has no positive steady states. 
		\item 
		If $G$ has exactly $j$ reactant complexes, where $2 \leq j \leq n$ (in particular, $n \geq 2$), and $G$ is full-dimensional, then every positive steady state (of \uline{every} mass-action system defined by $G$) is degenerate. 
	\end{enumerate}
\end{proposition} 
\begin{proof} 
	Assume $G$ has $n$ species, which we denote by $X_1, X_2, \dots,  X_n$, with exactly $j$ reactant complexes, for some $1\leq j \leq n$.  
	Let $\kappa^*$ be a vector of positive rate constants.
	As in~\eqref{eqn:ODEexpression}, we write the mass-action ODE system arising from $(G, \kappa^*)$ as follows:
	\begin{align} \label{eq:matrix-N}
	\begin{bmatrix} 
	dx_1/dt \\
	\vdots \\
	dx_n/dt \\
	\end{bmatrix}
	~=~
	N 
	\begin{bmatrix} 
	m_1 \\
	\vdots \\
	m_j \\
	\end{bmatrix}~
	~=:~
	\begin{bmatrix} 
	f_1 \\
	\vdots \\
	f_n \\
	\end{bmatrix}~
		,	
	\end{align} 
	where $N:=(N_{ij})$ is an $(n \times j)$-matrix (with entries in $\mathbb R$)
	and $m_1, \dots, m_j$ are distinct monic monomials in $x_1, \dots, x_n$ 
	(as $G$ has $n$ species and $j$ reactant complexes).

	We first prove part (1). In this case, the right-hand sides of the ODEs have the form $f_i = c_i \prod_{k=1}^{n}x_k^{a_k}$, with at least one $c_i\neq 0$. It follows that there are no positive steady states.

	We prove part (2). Assume that $G$ is full-dimensional 
	(the stoichiometric subspace 
	is $\mathbb{R}^n$) 	
	and that $2 \leq j \leq n$.  
	Let $x^*=(x_1^*,x_2^*,\ldots,x_n^*)$ be a positive steady state.  We must show $x^*$ is degenerate.
	
	We first consider the subcase when 
	the rank of the matrix $N$ is 
	at most $(n-1)$.
	%(strictly) less than $n$. 
	By Proposition~\ref{prop:rank-of-N}(1), every positive steady state is degenerate.

Now we handle the remaining subcase, when $N$ has rank $n$ (and hence, $N$ is $n\times n$).
Now, solving the steady-state equations $f_1= \dots = f_n=0$ can be accomplished by multiplying the expression in~\eqref{eq:matrix-N} by $N^{-1}$, which implies that every monomial $m_1,\dots, m_n$ evaluates to zero at steady state.  Hence, no positive steady states exist.
\end{proof}

\begin{proposition}[Networks with many reactants]\label{prop:fulldim-fullrank-PSS}
	If $G$ is a full-dimensional network with $n$ species %, $r$ reactions, 
	and exactly $j$ reactant complexes, where $j \geq n$, then:
		\begin{enumerate}
			\item There exists a vector of positive rate constants $\kappa^* $, such that the corresponding matrix $N$, as in~\eqref{eqn:ODEexpression}, has rank $n$.
			\item If there exists a vector of positive rate constants $\kappa^*$ such that the matrix $N$ {does} \uline{not} have rank $n$, then there exists 
			a vector of positive rate constants  
			$\kappa^{**} $ such that  
			$(G,\kappa^{**})$ has \uline{no} positive steady states.
	\end{enumerate} 
\end{proposition}
\begin{proof} 
	Assume $G$ is full-dimensional, with $n$ species, $r$ reactions (denoted by $y_1 \to y_1', \dots y_r \to y_r'$), 
	and exactly $j$ reactant complexes, where $j \geq n$.

	We begin with part (1). 
	Let $\kappa=(\kappa_1,\dots, \kappa_r)$ denote the vector of unknown rate constants (each $\kappa_i$ is a variable).
		Let $\widetilde N$ be the 
		$(n \times j)$
		matrix for $(G,\kappa)$ in the sense of $N$ in~\eqref{eqn:ODEexpression}. 
		More precisely, 
		the entries of $\widetilde N$ are $\mathbb Z$-linear combinations of the $\kappa_i$'s,
		such that, for every vector of positive rate constants $\kappa^* \in \R_{> 0}^r$, the evaluation $\widetilde N|_{\kappa=\kappa^*}$ is the matrix  $N$ as in~\eqref{eqn:ODEexpression} for $(G, \kappa^*)$.

	As $G$ is full-dimensional, there are no $\mathbb{R}$-linear relations among the $n$ rows of $\widetilde N$.
		Hence, the size-$n$ minors of $\widetilde N$
		define a (possibly empty) measure-zero subset $V \subseteq \R^r_{> 0}$.  
		Thus, $\R^r_{> 0} \smallsetminus V$ is nonempty, and every $\kappa^* \in \R^r_{> 0} \smallsetminus V$ yields a matrix $N=\widetilde N|_{\kappa=\kappa^*}$ with rank $n$.
	This proves part (1).

	For part (2), suppose that there exists $\kappa^* \in \R_{> 0}^r$ such that 
	the resulting matrix 
	$N$ has rank strictly less than $n$. 
	It follows that there is a linear relation:
	 \begin{align} \label{eq:lin-relation}
	 c_1f_{\kappa^*,1}+\dots + c_nf_{\kappa^*,n} ~=~0~,
	 \end{align} 
	where $c_1, \dots, c_n$ are real numbers -- not all $0$ -- and the $f_{\kappa^*,i}$ denote the right-hand sides of the mass-action ODEs for $(G,\kappa^*)$.  
	
	On the other hand, for unknown rate constants $\kappa$, as in the proof above for part~(1), 
	$c_1f_{\kappa,1}+\dots + c_nf_{\kappa,n} $ is \uline{not} the zero polynomial.  Thus, when we rewrite this expression as a sum over  $r$ reactions $y_i \to y_i'$ as follows: $c_1f_{\kappa,1}+\dots + c_nf_{\kappa,n}
	= d_1 \kappa_1 x^{y_1} +  \dots + d_r \kappa_r x^{y_r} $, where $d_i \in \mathbb{Z}$ for all $i$, we conclude that $d_i \neq 0$ for some $i$.  By relabeling reactions, if needed, we may assume that $i=1$.  
		
		Now consider the following vector of positive rate constants 
		$ \kappa_\epsilon^* := (\kappa_1^* + \epsilon , 
		\kappa_2^*, \dots, \kappa_r^*)$, for some $\epsilon>0$.  
Assume for contradiction that $(G,\kappa_\epsilon^*)$ has a positive steady state $x^*$. At steady state, 
$f_{\kappa^*_\epsilon,i} $ evaluates to $0$, for all $i$, and this yields the first equality here:
		\[
		0~=~ 
		\left( c_1f_{\kappa^*_\epsilon,1}+\dots + c_nf_{\kappa^*_\epsilon,n} \right) |_{x=x^*} 
		~=~ c_1f_{\kappa^*,1}|_{x=x^*}+\dots + c_nf_{\kappa^*,n}|_{x=x^*} + \epsilon d_1 x^{y_1}|_{x=x^*} 
		~=~
		\epsilon d_1 x^{y_1}|_{x=x^*} ,
		\]
		and the second and third equalities come from the fact that the mass-action ODEs are linear in the rate constants and from equation~\eqref{eq:lin-relation}, respectively.
		We obtain $x^{y_1}|_{x=x^*}=0$, which contradicts the fact that $x^*$ is a positive steady state. This concludes the proof.
\end{proof}

The next proposition returns to a topic from Proposition~\ref{prop:rank-of-N}, namely, networks with $n$ species and $n+1$ reactant complexes.

\begin{proposition}[Networks with $n+1$ reactants]\label{prop:fulldim_noPSS}
	Assume $G$ is a full-dimensional network, with $n$ species and exactly $n+1$ reactant complexes, 
	which we denote as follows:  
	\begin{align*}
y_{i1} X_1 + y_{i2} X_2 + \dots y_{in} X_n  \quad \quad {\rm for}~i=1,2,\dots,n+1~. 
	\end{align*}
	Let $A$ denote the $n \times n$ matrix obtained from the
	 $(n+1) \times n$ matrix
	$Y:=(y_{ij})$ by subtracting the last row from every row and then deleting the last row.
	\begin{enumerate}
		\item If $\operatorname{rank}(A) = n$, then $G$ is \uline{not} {nondegenerately} multistationary.   
		\item If $\operatorname{rank}(A) \leq n-1$, 
		 then there exists a vector of positive rate constants $\kappa^*$ such that $(G,\kappa^*)$ has \uline{no} positive steady states.
	\end{enumerate}
\end{proposition}

\begin{proof} 
{\bf Case 1:} $\operatorname{rank}(A) = n$.  Fix an arbitrary vector of positive rate constants $\kappa^*$.  We must show that $(G, \kappa^*)$ is not nondegenerately multistationary.  
Let $N$ denote the  $n \times (n+1)$ matrix defined by $(G, \kappa^*)$, as in~\eqref{eqn:ODEexpression}.  
We consider two subcases.

{\bf Subcase:}  $\operatorname{rank}(N) \leq n-1$.  In this subcase, Proposition~\ref{prop:rank-of-N}(1)
implies that every positive steady state of  $(G, \kappa^*)$ is degenerate, and so $(G, \kappa^*)$ is not nondegenerately multistationary.

{\bf Subcase:}  $\operatorname{rank}(N) = n$.   
 Part (2) of Proposition~\ref{prop:rank-of-N} pertains to this setting, so we can follow that proof.  
	In particular, 
equation~\eqref{eq:laurent-binomials} -- the $(n \times n)$ matrix $A$ there exactly matches the matrix $A$ here -- implies that the positive steady states are defined by 
a linear system of the form $Ay=\ln(\beta)$, where $y=(\ln(x_1), \dots, \ln(x_n))^{\top}$. 
Hence, as $\operatorname{rank}(A) = n$, we have at most one positive steady state and so $(G, \kappa^*)$ is not multistationary.  

{\bf Case 2:} $\operatorname{rank}(A) \leq n-1$.  We must show that there exists a choice of rate constants so that the resulting system has no positive steady states.  

Proposition~\ref{prop:fulldim-fullrank-PSS}(1)
implies that 
 there exists $\kappa^*$ such that the following holds:
 \begin{align} \label{eq:condition-for-kappa}
	  \text{the matrix $N$ defined by $(G,\kappa^*)$ has (full) rank $n$.}
\end{align}	 
Fix such a choice of $\kappa^*$. % for the rest of the proof.  
	 If $(G, \kappa^*)$ has no positive steady states, then we are done.  Therefore, for the rest of the proof, we assume that $(G, \kappa^*)$ admits a positive steady state.

In what follows, we need to consider additional vectors of positive rate constants (besides $\kappa^*$) and their corresponding matrices $N$, as  in~\eqref{eqn:ODEexpression}. Therefore, as in the proof of Proposition~\ref{prop:fulldim-fullrank-PSS}(1), let $\kappa=(\kappa_1,\dots, \kappa_r)$ (where $r$ is the number of reactions) denote the vector of unknown rate constants, and let $\widetilde N$ be the 
		$n \times (n+1)$
		matrix for $(G,\kappa)$ in the sense of $N$ in~\eqref{eqn:ODEexpression}, so that 
		for every vector of positive rate constants $\kappa^* \in \R_{> 0}^r$, the evaluation $\widetilde N|_{\kappa=\kappa^*}$ is the matrix  $N$ as in~\eqref{eqn:ODEexpression}. 

We now follow the ideas in the proof of Proposition~\ref{prop:rank-of-N}, part (2), with the difference being that we now consider unknown rate constants $\kappa$.  The mass-action ODEs for $(G,\kappa)$ are given by:
\begin{align*}
	\begin{bmatrix} 
	dx_1/dt \\
	\vdots \\
	dx_n/dt \\
	\end{bmatrix}
	~=~
	\widetilde N 
	\begin{bmatrix} 
	m_1 \\
	\vdots \\
	%m_n \\
	m_{n+1}\\
	\end{bmatrix}	~,
	\end{align*} 
where $m_1,\dots, m_{n+1}$ are distinct monic monomials in $x_1,x_2, \ldots,x_n$.   

Our next aim is to row-reduce $\widetilde N$ (over the field $\mathbb{Q}(\kappa_1,\dots, \kappa_r)$).  
Accordingly, for $1\leq k \leq n+1$, let $[B_k]$ denote the determinant of the matrix obtained from $\widetilde N$ by removing the $k$-th column. By construction, each $[B_k]$ is 
in $\mathbb{Z}[\kappa_1,\dots, \kappa_r]$.  

We claim that, for all $1 \leq k \leq n+1$, the polynomial $[B_k]$ is nonzero.  By symmetry among the monomials $m_i$, it suffices to show that $[B_{n+1}]$ is nonzero.  To show this, assume for contradiction that  $[B_{n+1}]=0$.  Then, $\widetilde N$ can be row-reduced to a matrix in which the last row has the form $(0,0,\dots, 0, \omega)$, where $0 \neq \omega \in \mathbb{Q}(\kappa_1,\dots, \kappa_r)$.  Now consider the evaluation at $\kappa=\kappa^*$.  By~\eqref{eq:condition-for-kappa}, the matrix $N = \widetilde N|_{\kappa=\kappa^*}$ has (full) rank $n$, so $\omega |_{\kappa=\kappa^*}$ is nonzero.  However, this implies that positive steady states of $(G,\kappa^*)$ satisfy $\omega|_{\kappa=\kappa^*} m_{n+1} = 0$, much like in~\eqref{eq:binomials-in-proof}.  Thus, $(G,\kappa^*)$ has no positive steady states, which is a contradiction, and hence our claim holds.

Next, as $[B_{n+1}]$ is nonzero, we can apply a version of Cramer's rule
to row-reduce $\widetilde N$ to the following matrix (where $I_n$ denotes the size-$n$ identity matrix):
		\begin{align*}
		{\widetilde N}' ~=~
	\left[
	\begin{array}{ccc|c}
		&  &  & (-1)^{n-1} \tfrac{[B_1]}{[B_{n+1}]} \\
		&  I_n &  & (-1)^{n-2} \tfrac{[B_2]}{[B_{n+1}]} \\
		&  &  & \vdots \\
		&  &  & (-1)^{0} \tfrac{[B_n]}{[B_{n+1}]}\\
		\end{array}
		\right]~.
		\end{align*}
 Thus, as in~\eqref{eq:binomials-in-proof}, the positive steady states are the 
 positive roots of the equations $m_i - \beta_i m_{n+1}=0$ (for $i=1,2,\dots,n$), where:
 \begin{align*}
	\beta_i ~:=~ (-1)^{n-i+1} \frac{[B_i]}
			{[B_{n+1]}}
	 \quad {\rm for}~ i=1,2,\dots,n~.
	\end{align*}
Thus, $\beta_i|_{\kappa= \kappa^*} >0$ (for all $i=1,2,\dots,n$), since $(G,\kappa^*)$ admits a positive steady state.  We conclude from this fact, plus the claim proven earlier (namely, that $[B_\ell] \neq 0$ for all $\ell$), that the following is an open subset of $\mathbb{R}^r_{>0}$ that contains $\kappa^*$:
\begin{align*}
	\Sigma ~:=~ \left\{ \bar\kappa \in \mathbb{R}^r_{>0}  ~:~ %\mid 
	 \beta_1|_{\kappa= \bar \kappa} >0,~\dots,~ \beta_n|_{\kappa= \bar \kappa} >0,~
 [B_1]|_{\kappa=\bar \kappa} \neq 0, \dots, [B_{n+1}]|_{\kappa= \bar \kappa} \neq 0 \right\}.
 \end{align*}

For the rest of the proof, we restrict our attention to rate constants, like $\kappa^*$, that are in $\Sigma$.  For such rate constants, like in~(\ref{eq:binomials-in-proof}--\ref{eq:linear-system-proof}), 
the positive steady states are the roots of the following equation
		\begin{align} \label{eq:linear-system-proof2}
		A 
		\begin{pmatrix} \ln(x_1) \\ \ln(x_2) \\ \vdots \\ \ln(x_n) \end{pmatrix}
		~=~
		\begin{pmatrix} \ln(\beta_1) \\ \ln(\beta_2) \\ \vdots \\ \ln(\beta_n)
		\end{pmatrix}
		~=:~ \ln (\beta)~.
		\end{align}

Next, as $\operatorname{rank}(A) \leq n-1$, there exists a nonzero vector $\gamma \in \mathbb{R}^n$ 
in the orthogonal complement of the column space of $A$.  
By relabeling the $m_i$'s (which permutes the columns of $\widetilde N$), if needed, 
we may assume that $\gamma_1 \neq 0$.  By construction of $\gamma$ and equation~\eqref{eq:linear-system-proof2}, we have $\langle \gamma,~ \ln(\beta) \rangle = 0$, which is readily rewritten as follows:
	\begin{align}\label{eqn:fornoPSS}
	\left( \frac{[B_1]}{[B_{n+1}]}\right)^{\gamma_1} 
	\dots 
	\left((-1)^{k+1} \frac{[B_k]}{ [B_{n+1}] }\right)^{\gamma_k}
	\dots
	\left((-1)^{n+1} \frac{ [B_n]}{ [B_{n+1}] }\right)^{\gamma_n}
	~=~1~.
	\end{align}

	For $\varepsilon>0$, let $\kappa^*_{\epsilon}$ denote the vector of rate constants obtained from $\kappa^*$ by scaling by $(1+ \varepsilon)$ all rate constants of reactions in which the reactant generates the monomial $m_1$.  
	As $\Sigma$ is an open set, $\kappa^*_{\epsilon}\in \Sigma$ for $\varepsilon$ sufficiently small.  
	Also, by construction, the matrix $\widetilde N |_{\kappa = \kappa^*_{\varepsilon}}$ is obtained from 
	$\widetilde N |_{\kappa = \kappa^*}$ by scaling the first column by  $(1+ \varepsilon)$.  
	So, for $2 \leq i \leq n+1$, we have $[B_i]  |_{\kappa = \kappa^*_{\varepsilon}} = (1+ \epsilon) [B_i]  |_{\kappa = \kappa^*}$.  
	
	Thus, by replacing $\kappa^*$ by $\kappa^*_{\varepsilon}$, the left-hand side of equation~\eqref{eqn:fornoPSS}  is scaled by 
	$(1+\varepsilon)^{-\gamma_1}$, and so there exists $\varepsilon>0$ for which equation~\eqref{eqn:fornoPSS} does not hold (when evaluated at $\kappa = \kappa^*_{\varepsilon}$).	
	Hence, this vector $\kappa^*_{\varepsilon}$ yields a mass-action system $(G,\kappa^*_{\varepsilon})$ with no positive steady states, as desired.
\end{proof}

Proposition~\ref{prop:fulldim-fullrank-PSS} implies that for networks with at least $n$ reactant complexes (where $n$ is the number of species), some choice of rate constants yields a matrix $N$ with (full) rank $n$.  
Our next result shows that when this condition holds (even for networks with fewer reactants), 
every species appears in at least one reactant complex.

We introduce the following shorthand (which we use in several of the next results): a complex 
$y_{\ell 1}X_1+ y_{\ell 2}X_2 + \dots + y_{\ell n}X_n$ {\em involves} species $X_i$ if $y_{\ell i} \neq 0$.  For instance, $X_1+X_2$ involves $X_2$, but $X_1+X_3$ does not.  

\begin{lemma}[Reactants involve all species] \label{lem:involve-all-species}
Let $G$ be a full-dimensional reaction
network with $n$ species, let $\kappa^*$ be a vector of positive rate constants, and let $N$ be the matrix for $(G,\kappa^*)$, as in~\eqref{eqn:ODEexpression}.  If $\operatorname{rank}(N) = n$ and $(G,\kappa^*)$ has a positive steady state, then 
for every species $X_i$, at least one reactant complex of $G$ involves $X_i$.
\end{lemma}
\begin{proof}
We prove the contrapositive.  Assume that there is a species $X_i$ such that for every reactant complex $a_1X_1+ a_2 X_2 + \dots +a_nX_n$ we have $a_i = 0$.  Then, by Lemma~\ref{lem:hungarian}, the right-hand side of the mass-action ODE for $X_i$, which we denote by $f_i$, is a sum of monomials, all of which have positive coefficients.  But $(G,\kappa^*)$ has a positive steady state, so $f_i$ must be $0$.  We conclude that the $i$-th row (of the $n$ rows) of $N$ is the zero row and so $\operatorname{rank}(N)  \leq  n-1$. 
\end{proof}

\subsection{Networks with conservation laws} \label{sec:conservation-laws}

The following result is similar to several results in the prior subsection, but pertains to networks that are not full-dimensional.

\begin{proposition}[Networks with conservation laws and few reactants]
	\label{prop:min-num-reac-with-k-cons-law}
	Let $G$ be a reaction network with $n \geq 3$ species.  
	Assume that $G$ is $(n-k)$-dimensional, where $k \geq 1$ (so, $G$ has $k$ conservation laws).
	% (more precisely, the stoichiometric subspace of $G$ has dimension $n-k$).
	 If $G$ has exactly $j$ reactant complexes, for some $j \in \{2,3,\dots, n-k\}$, 
	then every positive steady state (of \uline{every} mass-action system defined by $G$) is degenerate.
\end{proposition}
\begin{proof}
	We mimic the proofs of Propositions~\ref{prop:rank-of-N}(1) and \ref{prop:2-or-3-rxns3species}. 
	Let $\kappa^*$ be a vector of positive rate constants. Let $N$ be an $(n \times j)$ matrix defined, as in~\eqref{eqn:ODEexpression}, by $(G,\kappa^*)$:
	
		\begin{align} \label{eq:matrix-N-again}
	\begin{bmatrix} 
	dx_1/dt \\
	\vdots \\
	dx_n/dt \\
	\end{bmatrix}
	~=~
	N 
	\begin{bmatrix} 
	m_1 \\
	\vdots \\
	m_j \\
	\end{bmatrix}~
	~=:~
	\begin{bmatrix} 
	f_1 \\
	\vdots \\
	f_n \\
	\end{bmatrix}~
		,	
	\end{align} 
	where
	$m_1, \dots, m_j$ are distinct monic monomials in $x_1, \dots, x_n$. 
	
	We consider two cases.  First assume that $\operatorname{rank}(N) \leq n-k-1$.   Then the polynomials $f_i$ span a subspace of dimension $\leq n-k-1$ and hence 
	the Jacobian matrix -- even before evaluating at a positive steady state -- has rank $\leq n-k-1$.  Every positive steady state is therefore degenerate.
	
Consider the remaining case: $\operatorname{rank}(N) = n-k$ (so, $j=n-k$).
In this case, multiplication by $N$  %(on the left) 
defines an injective map $\mathbb{R}^{n-k} \to \mathbb{R}^n$.  
Hence, by~\eqref{eq:matrix-N-again}, the steady-state equations $f_1= \dots = f_{n}=0$ 
imply the monomial equations $m_1=\dots= m_{j}=0$.  Thus, there are no positive steady states.
\end{proof}

The next result concerns networks with $n-1$ conservation laws, that is, one-dimensional networks.

\begin{proposition}[One-dimensional networks] \label{prop:1-d-no-coexistence}
	Let $G$ be a one-dimensional reaction network, and let $\kappa^*$ be a vector of positive rate constants. 
	If $(G,\kappa^*)$ has ACR, then  $(G,\kappa^*)$  is \uline{not} nondegenerately multistationary.
\end{proposition}

\begin{proof}
	Assume that $G$ 
	is one-dimensional, with $n$ species.
		Thus, $G$ has $n-1$ linearly independent conservation laws.
 Let $\kappa^*$ be a vector of positive rate constants for which there is ACR.  
 We may assume that the ACR species is $X_1$ (by relabeling species, if needed).   Let $f_1,\dots, f_n$ denote the right-hand sides of the mass-action ODEs arising from $(G,\kappa^*)$.
 
 Let $x^*=(x_1^*, \dots, x_n^*)$ denote an arbitrary positive steady state of $(G,\kappa^*)$. (The ACR-value is $x^*_1$.)
 Let $P_{x^*}$ denote the (one-dimensional) stoichiometric compatibility class that contains $x^*$.  
It suffices to show that (1) $x^*$ is the unique positive steady state in $P_{x^*}$ or (2) $x^*$ is degenerate.

 	 We consider two cases.

%\noindent
	{\bf Case (a):} $X_1$ is \uline{not} a catalyst-only species (in some reaction of $G$).  
	This implies that $f_2,\dots, f_n$ are all scalar multiples of $f_1$, and that 
	the compatibility class 
	$P_{x^*}$ is defined by $n-1$ conservation laws of the form 
	$x_j=a_j x_1+b_j$, where $a_j,b_j \in \mathbb{R}$, for $j\in \{2,3,\dots, n\}$.
	By substituting these $n-1$ relations into $f_1$, we obtain a univariate polynomial in $x_1$, which we denote by $h$.  If $h$ has multiple positive roots, then there is no ACR, which is a contradiction.  
	If, on the other hand, $h$ does not have multiple positive roots, then
	$P_{x^*}$ does not contain multiple positive steady states (that is, $x^*$ is the unique positive steady state in $P_{x^*}$).

	{\bf Case (b):} $X_1$ is a catalyst-only species in all reactions of $G$.	
	In this case, $f_1=0$, and $x_1=x_1^*$ is a conservation law of $G$, and it is one of the defining equations of the compatibility class 
	$P_{x^*}$.  By relabeling species $X_2,\dots, X_n$, if needed, we may assume that 
	$X_2$ is not a catalyst-only species (as $G$ is one-dimensional). 
	Thus, we can ``extend'' the conservation law $x_1=T$ to a ``basis'' of $n-1$ conservation laws 
	 that define the compatibility class
	$P_{x^*}$,
	by appending $n-2$ conservation laws of the form 
	 $x_j=a_jx_2+b_j$, where $a_j,b_j \in \mathbb{R}$, for $j\in \{3,4,\dots, n\}$. 
	 
	 Next, we substitute these $n-2$ conservation relations into $f_2$, which yields a polynomial in $x_1$ and $x_2$, which we denote by $g$. Consider the following set, which is the positive variety of $g$ in $\mathbb{R}^2_{>0}$ (the values of $x_3,\dots, x_n$ are free, so we ignore them):
	\begin{align} \label{eq:variety-of-g}		 
	\Sigma ~:=~ \{x\in \mathbb{R}^2_{>0} \mid g(x_1,x_2)=0 \}.
	 \end{align}
	 By construction and the fact that there is ACR in $X_1$, the set $\Sigma$ is contained in the hyperplane  (line) $x_1=x_1^*$, and so is either one-dimensional or zero-dimensional.
	 %has dimension at most $n-1$.  
We consider these two subcases separately.  
First, assume that $\Sigma$ is one-dimensional. In this subcase, $\Sigma$ equals the subset of the hyperplane $x_1=x_1^*$ in the positive quadrant $\mathbb{R}^2_{>0}$, and so the compatibility class
	$P_{x^*}$ consists entirely of positive steady states.  The Inverse Function Theorem now implies that every positive steady state of
	$P_{x^*}$ (in particular,
	 $x^*$) is degenerate.

	 Consider the remaining subcase, in which $\Sigma$ is zero-dimensional
	(that is, $\Sigma$ consists of finitely many points).   
It follows that $g$ is either non-negative on  $\mathbb{R}^2_{>0}$ or non-positive on  $\mathbb{R}^2_{>0}$, and so $f_2$ is either non-negative on $P_{x^*}$ or non-positive on $P_{x^*}$. 
Consequently, as every $f_i$ is a scalar multiple of $f_2$,
the steady state $x^*$ is degenerate.
\end{proof}

\subsection{Bimolecular networks} \label{sec:bimol-prelim-results}
We begin this subsection with a result that clarifies how the polynomials arising in mass-action ODEs are constrained when the network is bimolecular.

\begin{lemma}[Bimolecular networks]\label{lem:bimol-Descartes}
	Consider a bimolecular mass-action system $(G,\kappa^*)$ with $n$ species.
	Let $f_i$ be the right-hand side of the mass-action ODE for species $X_i$ (for some  $1 \leq i \leq n$). 
	Fix positive values $a_j >0$ for all $j \in \{1,2,\dots, n\} \smallsetminus \{i\}$.
	%$1 \leq j \leq n$ with $j \neq i$.  
	Let $g_i$ denote the univariate polynomial obtained by evaluating $f_i$ at $x_j=a_j$ for all $j \in \{1,2,\dots, n\} \smallsetminus \{i\}$.  If the polynomial $g_i$ is nonzero, then $g_i$ has at most one sign change and hence has at most one positive root.
\end{lemma}
\begin{proof}
Let $g_i$ denote the nonzero polynomial obtained by evaluating $f_i$ at $x_j=a_j$ for all $j\neq i$.  
Several properties of $g_i$ arise from the fact that $G$ is bimolecular:
(1) $\deg(g_i) \leq 2$, (2) the coefficient of $x_i^2$ is non-positive, and (3) the constant coefficient is non-negative.  Thus, $g_i$ has at most one sign change, and so Descartes' rule of signs implies that $g_i$ has at most one positive root.  
\end{proof}

The next two results 
pertain to bimolecular mass-action systems in which the right-hand side of some ODE vanishes (Propositions~\ref{lem:when-ODE-is-0}) or vanishes when evaluated at an ACR-value (Proposition~\ref{lem:when-evaluated-ODE-is-0}). We motivate these results through the following example.

	\begin{example}[Enlarged Shinar-Feinberg network] \label{ex:preview-to-2-lemmas}
		A common way to  construct a network with an ACR species (e.g., $A$) is through the existence of an $f_i$ that becomes zero when we substitute the ACR-value in place of the species. We illustrate this idea through the following network:
		\[
		G~=~
		\{A+B \xrightarrow{\kappa_1} 2B ,~ B\xrightarrow{\kappa_2} A,~ 0 \xleftarrow{\kappa_3} B+C \xrightarrow{\kappa_4} 2B,~ 0\xrightarrow{\kappa_5}C\}~.
		\]
		This network is constructed from a well-studied network first introduced by Shinar and Feinberg~\cite{ACR} 
		by adding three reactions involving a new species ($C$).
		We examine the mass-action ODE for $B$:
		\begin{align*}
		&\frac{dx_2}{dt}~=~ 
		\kappa_1x_1x_2-\kappa_2 x_2 - \kappa_3 x_2x_3+ \kappa_4 x_2x_3 = x_2(\kappa_1 x_1-\kappa_2) + x_2x_3(-\kappa_3+\kappa_4)
		~=~ g+h~ =:~ f_2~,
		\end{align*}
		where $g:=  x_2(\kappa_1 x_1-\kappa_2)$ (which is the right-hand side of the ODE for $X_2$ in the original Shinar-Feinberg network) and 
		$h:= x_2x_3(-\kappa_3+\kappa_4)$ (arising from the additional reactions, involving $X_3$).

		Assume $\kappa_3=\kappa_4$. It is easy to check that $(G,\kappa)$ has 
		a positive steady state and also has 
		ACR in species $X_1$ with ACR-value $\alpha = \kappa_2/\kappa_1$. Also, observe that $f_2|_{x_1=\alpha}=0$, as a result of the equalities $g|_{x_1=\alpha}=0$ and $h=0$ (which is due to the equality $\kappa_3=\kappa_4$).
		
		The next two results characterize which reactions can exist in such a situation. More precisely:
		\begin{itemize}
			\item Proposition \ref{lem:when-ODE-is-0} gives conditions that hold when a mass-action ODE is zero (effectively characterizing what reactions 
			can yield $h=0$ in this case). 
			\item Proposition \ref{lem:when-evaluated-ODE-is-0} gives conditions that hold when a mass-action ODE is zero when evaluated at the ACR-value (effectively characterizing what reactions 
			can 
				yield $f_2|_{x_1=\alpha}=0$ in this case, involving a decomposition like the one we observed above: $f_2|_{x_1=\alpha}=g|_{x_1=\alpha} +h$).
		\end{itemize} 
	\end{example}

The next result uses the following notation:

\begin{notation}[Empty complex] \label{notation:dummy-variable}
	We introduce the dummy variable $X_0:=0$, so that (for instance) $X_0$ is the empty complex and $X_i+X_0 := X_i$ for any species $X_i$.
\end{notation}

The following result clarifies which reactions can exist if some mass-action ODE is zero.

\begin{proposition}[When $f_i$ is zero] \label{lem:when-ODE-is-0}
	Let $G$ be a bimolecular reaction network with $n$ species $X_1, X_2, \dots, X_n$.  Fix $1 \leq i \leq n$.
	Let $\kappa^*$ denote a vector of positive rate constants for $G$, 
	and 
	let $f_i$ denote the right-hand side of the mass-action ODE for species $X_i$ in the system $(G, \kappa^*)$. If $f_i$ is the zero polynomial, then 
	the set of reactions of $G$ in which $X_i$ is a non-catalyst-only species is a (possibly empty) subset of the reactions listed here (where our use of $X_0$ follows Notation~\ref{notation:dummy-variable}):  
	% ------------
	% LIST
	% ------------
	\begin{enumerate}
		\item the reactions of the form  $X_i +X_j \to 2X_i$ %$\{X_i \to 2X_i \} $ 
		(and we denote the rate constant by $\kappa^*_{1,j}$), where $j\in \{0,1,\ldots n\} \smallsetminus \{i\}$,
		\item the reactions of the form  $X_i + X_j \to \star$ (with rate constant $\kappa^*_{2,j,\ell}$, where $\ell$ is an index for such reactions), where $j\in \{0,1,\ldots n\} \smallsetminus \{i\}$ and $\star$ is any complex that does not involve $X_i$,
	\end{enumerate}
	and, additionally, the following relationships among the rate constants hold:
	\begin{align} \label{eq:relation-among-rate-constants-when-f-is-zero}
	\kappa^*_{1,j} ~=~
	\sum_\ell\kappa^*_{2,j,\ell} ~
	\quad \quad
	{\rm for~all~} j\in \{0,1,\ldots n\} \smallsetminus \{i\}~,
	\end{align}
	where a rate constant is set to 0 if the corresponding reaction is not in $G$.
\end{proposition}
%-----------------------
% PROOF
%-----------------------
\begin{proof}
	Let $\kappa^*$ be a vector of positive rate constants for a bimolecular network $G$ with $n$ species, 
	and let $f_i$ be the right-hand side of $(G,\kappa^*)$ for the species $X_i$. 
	Let $\Sigma$ denote the set of reactions of $G$ in which $X_i$ is a non-catalyst-only species.  Reactions {\em not} in $\Sigma$ do not contribute to $f_i$, so we ignore them for the rest of the proof.
	
	We claim that for all reactions in $\Sigma$, the reactant complex is {\em not} one of the following 5 types: $0$, $X_j$, $X_j + X_{j'}$, $2X_j$, $2X_i$ for any $j,j' \in \{1,2,\dots, n \} \smallsetminus \{i\}$.  Indeed, any of the first $4$ types of complexes would yield a constant term in $f_i$ (when viewed as a polynomial in $x_i$) consisting of a sum of monomials with positive coefficients; similarly, the last type ($2X_i$) would yield a negative $x_i^2$ term (the fact that $G$ is bimolecular is used here).  However, $f_i$ is zero, so the claim holds.

	It follows that, for every reaction in $\Sigma$, the reactant complex either is $X_i$ or has the form $X_i+X_j$ for some $j \in \{1,2,\dots, n \} \smallsetminus \{i\}$.  It is straightforward to check that all possible such reactions (in which $X_i$ is a non-catalyst-only species) are listed in the proposition.  Next, reactions of type~(1) in the proposition contribute positively to $f_i$, while those of type~(2) contribute negatively, as follows:
	\begin{align} \label{eq:f_i-in-detail}
	f_i ~=~
	\left(
	\kappa^*_{1,0} - \sum_\ell \kappa^*_{2, 0, \ell} 
	\right) x_i 
	~+
	\sum_{j\in  \{1,2,\dots, n \} \smallsetminus \{i\}}
	\left(
	\kappa^*_{1,j} - \sum_\ell \kappa^*_{2,j,\ell} 
	\right)
	x_i x_j~.
	\end{align}
	As $f_i=0$, the coefficient of $x_i$ and the coefficient of each $x_{ij}$ in~\eqref{eq:f_i-in-detail} must be $0$, which yields the desired equalities~\eqref{eq:relation-among-rate-constants-when-f-is-zero}.
\end{proof}

Proposition~\ref{lem:when-ODE-is-0} concerns general (bimolecular) mass-action systems, and now we consider those with ACR.  The next result characterizes which reactions can exist if some mass-action ODE becomes zero when evaluated at the ACR-value.

\begin{proposition}[When $f_i$ is zero at the ACR-value] \label{lem:when-evaluated-ODE-is-0}
	Let $G$ be a bimolecular reaction network with species $X_1, X_2, \dots, X_n$, where $n \geq 2$.
	Let $\kappa^*$ denote a vector of positive rate constants. 
	Assume that the mass-action system $(G, \kappa^*)$ has ACR in species $X_1$ with ACR-value $\alpha >0$. 
	Fix $2 \leq i \leq n$.  
	Let $f_i$ denote the right-hand side of the mass-action ODE for species $X_i$ in the system $(G, \kappa^*)$.  
	If $f_i\neq 0$ and $f_i|_{x_1=\alpha}$ is the zero polynomial, then 
	the set of reactions of $G$ in which $X_i$ is a non-catalyst-only species 
	is a nonempty subset of the following reactions (the same as the ones in Proposition~\ref{lem:when-ODE-is-0}): 
	% ------------
	% LIST
	% ------------
	\begin{enumerate}
		\item the reactions of the form  $X_i +X_j \to 2X_i$ %$\{X_i \to 2X_i \} $ 
		(and we denote the rate constant by $\kappa^*_{1,j}$), where $j\in \{0,1,\ldots n\} \smallsetminus \{i\}$,
		\item the reactions of the form  $X_i + X_j \to \star$ (with rate constant $\kappa^*_{2,j,\ell}$, where $\ell$ is an index for such reactions), where $j\in \{0,1,\ldots n\} \smallsetminus \{i\}$ and $\star$ is any complex that does not involve $X_i$.
	\end{enumerate}
	Additionally, the following 
	relationship between the ACR-value $\alpha$ and the rate constants holds: 
	\begin{align} \label{eq:relation-rate-constants}
	\alpha ~=~ \frac{ \left( \sum_{\ell} \kappa^*_{2,0, \ell} \right) - \kappa^*_{1,0}}{  \kappa^*_{1,1} -  \left( \sum_\ell \kappa^*_{2,1,\ell} \right)  }~. 
	\end{align}
	In particular, the numerator and denominator of~\eqref{eq:relation-rate-constants} are nonzero. 
	Finally, if $n\geq 3$, then the following 
	relationships among the rate constants hold:
	\begin{align} \label{eqn:kappa-relation} 
	\kappa^*_{1,j} ~=~
	\sum_\ell\kappa^*_{2,j,\ell} ~
	\quad \quad
	{\rm for~all~} j\in \{2,3,\ldots n\} \smallsetminus \{i\}~.
	\end{align}
	(In equations~\eqref{eq:relation-rate-constants}--\eqref{eqn:kappa-relation}, a rate constant is set to 0 if the corresponding reaction is not in $G$.)
\end{proposition}

\begin{proof}
	Assume that $f_i$ is nonzero, but $f_i|_{x_1=\alpha}$ is zero.  Using properties of polynomial rings over a field, it follows that $(x_1- \alpha)$ divides $f_i$. 
	From the fact that $G$ is bimolecular, we conclude that:
	\begin{align} \notag
	f_i~&=~(x_1 - \alpha) \left( \beta x_i + \gamma + \sum_{j \in [n] \smallsetminus \{i\}} \delta_j x_j \right) \\
	\label{eq:RHS-evaluated}
	&=~  \beta x_1 x_i 
	+ \gamma x_1
	+ \left( \sum_{j \in [n] \smallsetminus \{i\}} \delta_j x_1 x_j \right)
	-\alpha \beta x_i 
	-\alpha \gamma 
	- \left( \sum_{j \in [n] \smallsetminus \{i\}} \alpha \delta_j x_j \right)~,
	\end{align}
	for some real numbers $\beta, \gamma, \delta_j$, at least one of which is nonzero.  
	
	In the right-hand side of~\eqref{eq:RHS-evaluated}, the variable $x_i$ does not appear in any of the following monomials (here the hypothesis $i\neq 1$ is used): 
	\[
	\gamma x_1, \quad -\alpha \gamma, \quad
	\delta_j x_1 x_j, \quad -\alpha \delta_j x_j~, 
	\]
	for
	$ j \in \{1,2,\dots,n\} \smallsetminus \{i\} $,  
	so Lemma~\ref{lem:hungarian} implies that the coefficients of these monomials must be non-negative.  
	Since $\alpha>0$, we conclude that $\gamma=0$ and $\delta_j=0$ 
	(for all $j \in \{1,2,\dots,n\} \smallsetminus \{i\} $).

	Thus, using~\eqref{eq:RHS-evaluated}, we have 
	$f_i=  \beta x_1 x_i - \alpha \beta x_i$, for some $\beta \in \mathbb{R} \smallsetminus \{0\} $.  Next, we investigate which reactions contribute to the two  monomials in $f_i$.  
	For $ \beta x_1 x_i$, the contributing reactions 
	have the form $X_1+X_i\to 2X_i$ and $X_1 + X_i \to \star$, where $\star$ does not involve $X_i$. The first reaction contributes positively,  
	while the second type contributes negatively. 
	Let $\kappa^*_{1,1}$ be the reaction rate constant for $X_1+X_i\to 2X_i$ (as in the statement of the lemma) and $\kappa^*_{2,1,\ell}$ be the rate constant for reactions of type $X_1 + X_i \to \star$, where $\ell$ is an index for all the reactions of this type. We conclude that 
	\begin{align} \label{eq:sum-beta}
	\kappa^*_{1,1} - \sum_\ell \kappa^*_{2,1,\ell} ~= ~ \beta~.
	\end{align}

	Similarly, the monomial $-\alpha \beta x_i$ in $f_i$ comes from reactions of the form $X_i \to 2X_i$, which contributes positively, 
	and $X_i \to \star$, which contributes negatively, where $\star$ is a complex that does not involve $X_i$. 
	Hence, 
	\begin{align} \label{eq:sum-alpha-beta}
	\kappa^*_{1,0} - 
	\sum_\ell \kappa^*_{2,0,\ell}  ~=~ - \alpha \beta~.
	\end{align} 
	Now the 
	equations~\eqref{eq:sum-beta}
	and~\eqref{eq:sum-alpha-beta} together imply 
	the desired equality~\eqref{eq:relation-rate-constants}.
	
	Next, let $\Sigma$ denote the set of reactions of $G$ in which $X_i$ is a non-catalyst-only species.  We showed above that $\Sigma$ contains a (nonempty) subset of reactions with rate constants labeled by $\kappa^*_{1,0}$, $\kappa^*_{1,1}$, $\kappa^*_{2,0,\ell}$, $\kappa^*_{2,1,\ell}$.  Let $\Sigma' \subseteq \Sigma$ denote the remaining reactions, and let $G'$ denote the subnetwork defined by the reactions in $\Sigma'$.  
	Let $\kappa'$ be obtained from $\kappa^*$ by restricting to coordinates corresponding to reactions in $\Sigma'$.  
	By construction, the mass-action ODE of $(G', \kappa')$ for species $X_i$ has right-hand side equal to $0$.
	So, 
	Proposition~\ref{lem:when-ODE-is-0} applies (where reactions arising from $j=0,1$ in that proposition are absent from $G'$ by construction), 
	and yields two conclusions.  
	First, 
	$\Sigma'$ is a subset of the reactions listed in Proposition~\ref{lem:when-evaluated-ODE-is-0} (specifically, with $j \not = 0,1$), and so $\Sigma$ is a subset of the full list (including $j = 0,1$).
	Second, 
	the equations~\eqref{eq:relation-among-rate-constants-when-f-is-zero} hold (for $j \not = 0,1$), which are the desired equalities~\eqref{eqn:kappa-relation}. 
\end{proof}

\begin{remark} \label{rmk:lemma}
	The reactions listed in 
	Propositions~\ref{lem:when-ODE-is-0} and~\ref{lem:when-evaluated-ODE-is-0} (the lists are the same) are not reversible.  Hence, if $G$ is a {\em reversible} network satisfying the hypotheses of either proposition, then $X_i$ is a catalyst-only species in every reaction of $G$.
\end{remark}

	\begin{example_contd}[\ref{ex:preview-to-2-lemmas}] \label{ex:after-2-lemmas}
		We revisit the enlarged Shinar-Feinberg network, $
		G=
				\{A+B \xrightarrow{\kappa_1} 2B ,~ B\xrightarrow{\kappa_2} A,~ 0 \xleftarrow{\kappa_3} B+C \xrightarrow{\kappa_4} 2B,~ 0\xrightarrow{\kappa_5}C\}.$
		Recall that, when $\kappa_3=\kappa_4$, the mass-action system $(G,\kappa)$ has ACR in $X_1$ with ACR-value $\alpha= \kappa_2/\kappa_1$, and that $f_2|_{x_1=\alpha}=0$.  In the notation of Proposition~\ref{lem:when-evaluated-ODE-is-0}, the rate constants of reactions in which $X_2$ is non-catalyst-only are:
		\[
		\kappa^*_{1,1}=\kappa_1, ~\quad
		\kappa^*_{1,0}=\kappa_2, ~\quad
		\kappa^*_{2,3,1}=\kappa_3,~\quad
		\kappa^*_{1,3}=\kappa_4~.  
		\]
		Now the formula in Proposition~\ref{lem:when-evaluated-ODE-is-0} for the ACR-value~\eqref{eq:relation-rate-constants} exactly yields the ACR-value computed earlier: $\alpha= \kappa_2/\kappa_1$, and  
		the relationship among rate constants~\eqref{eqn:kappa-relation} recapitulates $\kappa_3=\kappa_4$.
	\end{example_contd}

\begin{remark} \label{rem:robust-ratio}
The formula for the ACR-value, in~\eqref{eq:relation-rate-constants}, is related to the concept of ``robust ratio'' introduced by Johnston and Tonello~\cite{tonello2017network}.
\end{remark}

\subsection{Three reversible reactions are necessary for multistationarity} \label{sec:multismallnetworks} 
 Recall that, in~\eqref{eq:3-rev-rxn-rates}, we saw an instance of a (nondegenerately) multistationary, bimolecular network that consists of $3$ pairs of reversible reactions.  In this subsection, we prove that bimolecular networks with fewer pairs of reversible reactions are non-multistationary
%In this subsection, we prove that a bimolecular, reversible network must contain at least $3$ pairs of reversible reactions in order to be multistationary 
(Theorem~\ref{prop:2-rev-rxn-no-mss}).  
Our proof of Theorem~\ref{prop:2-rev-rxn-no-mss} 
requires 
%a result that states that 
%one-dimensional, bimolecular networks can only be multistationary if they are degenerately so
%(Lemma~\ref{lem:1-dim-no-mss}), 
%plus two 
several
supporting lemmas on one-dimensional networks.

%one-dimensional, bimolecular networks can only be multistationary if they are degenerately so
%(Theorem~\ref{prop:1-dim-no-mss}), 
%and that bimolecular networks consisting of one or two reversible reactions are non-multistationary
%(Theorem~\ref{prop:2-rev-rxn-no-mss}). To prove these two results, we first need two supporting lemmas. 

For the next lemma, recall from Section~\ref{sec:mass-action} that ${\rm cap}_{pos}(G)$ 
(respectively,  ${\rm cap}_{nondeg}(G)$)
denotes the maximum possible number of positive (respectively, nondegenerate and positive) steady states of a network $G$. In Lemma~\ref{lem:lin-tang-zhang} below, part~(1) was conjectured by Joshi and Shiu~\cite{joshi2017small} and then proved by Lin, Tang, and Zhang~\cite[Theorem 4.3]{lin-tang-zhang} (see also~\cite{pantea-voitiuk}). 
% \badalnote{add citation to Casian's paper} DONE
Part~(2) is due to Tang and Zhang~\cite[Theorem~6.1]{tang-zhang}.
%\badalnote{added reference.}
%In the following result, part~1 is due to Lin, Tang, and Zhang~\cite[Theorem 4.3]{lin-tang-zhang}, 
%and part~2 is due to Tang and Zhang~\cite[Theorem~6.1]{tang-zhang}.

\begin{lemma} \label{lem:lin-tang-zhang}
Let $G$ be a one-dimensional reaction network.  
\begin{enumerate}
	\item If $G$ is multistationary and ${\rm cap}_{pos}(G) < \infty$, then $G$ has an embedded one-species network with arrow diagram $(\leftarrow, \to)$ and another with arrow diagram $(\to, \leftarrow)$. %\badalnote{corrected the statement.}
	\item If ${\rm cap}_{pos}(G) < \infty$, then ${\rm cap}_{nondeg}(G) = {\rm cap}_{pos}(G)$.
\end{enumerate}
\end{lemma}

Joshi and Shiu showed that the network $G=\{ 0 \leftarrow A \to 2A\}$ is the only one-species, bimolecular network for which ${\rm cap}_{pos}(G) =\infty$~\cite{atoms_multistationarity}. 
The following lemma generalizes this result from one-species networks to one-dimensional networks.  

\begin{lemma}[One-dimensional bimolecular networks with infinitely many steady states]
\label{lem:1-d-infinite-steady-states}
Let $G$ be a one-dimensional and bimolecular reaction network with $n$ species. The following are equivalent:
	\begin{enumerate}
	\item ${\rm cap}_{pos}(G) = \infty$.
	\item Up to relabeling species, $G$ is one of the following networks:
		\begin{enumerate}
		\item $\{2X_1 \leftarrow X_1+X_2 \to 2X_2 \}$, 
		\item $\{X_1 \to 2X_1\} \cup \Sigma$, where $\Sigma$ consists of at least one reaction from the following set: 
	\[
\{0 \leftarrow X_1 \} \cup \{X_i \leftarrow X_1 +  X_i \mid i=2,3,\dots, n\}~.
	\]
		\end{enumerate}
	\end{enumerate}
Additionally, for the networks listed above in 2(a) and 2(b), 
every positive steady state (of every mass-action system arising from the network $G$) is degenerate.
\end{lemma}

\begin{proof}
Let $G$ be a one-dimensional, bimolecular reaction network.  
Up to relabeling species, the one-dimensional stoichiometric subspace is spanned by one of the following seven vectors:
\begin{footnotesize}
	\begin{align}
	%------------------
	\label{eq:line-1}
	%------------------
	& (1,0,0,\dots, 0)~,  (1,-1,0,0,\dots, 0)~,\\
	%------------------
	\label{eq:line-2}
	%------------------
	& (1,1,0,0,\dots, 0)~, 
	 (1,-2, 0,0,\dots, 0)~, 
	(1,1,-1,0,0,\dots, 0)~, 
	 (1,1,-2, 0,0,\dots, 0)~, 
	 (1,1,-1,-1, 0,0,\dots, 0)~.
	\end{align}
\end{footnotesize}
We first consider the case when the stoichiometric subspace is spanned by one of the five vectors listed in~\eqref{eq:line-2}. The network $G$ is then a subnetwork of one of the following networks (where we use $A,B,C,D$ in place of $X_1,X_2,X_3,X_4$ for ease of notation);
	\begin{align*}
	\{0 \leftrightarrows A+B\}~,~
	\{A \leftrightarrows 2B\}~,~
	\{A+B \leftrightarrows C\}~,~
	\{A+B \leftrightarrows 2C\}~,~
	\{A+B \leftrightarrows C+D\}~.
	\end{align*}
A direct calculation shows that the deficiency of $G$ is 0, so the deficiency-zero theorem (Lemma~\ref{lem:thm:def-0})
implies that $G$ is not multistationary.  In particular, ${\rm cap}_{pos}(G) < \infty$.

Having shown that the case of~\eqref{eq:line-2} 
is consistent with Lemma~\ref{lem:1-d-infinite-steady-states}, we now consider the remaining two cases, from~\eqref{eq:line-1}, separately.
%We first consider the case when 
First, assume the stoichiometric subspace of $G$ is spanned by $(1,0,0,\dots, 0)$.  It follows that the reactions of $G$ form a subset of the following $2n+4$ reactions:
	%------------------
	\begin{align*}
	%------------------
    &0 \stackrel[k_0]{m_0}{\leftrightarrows} X_1 
	\stackrel[m_1]{\ell_0}{\leftrightarrows} 2X_1 \quad \quad 
    0 \stackrel[k_1]{\ell_1}{\leftrightarrows} 2X_1 \\
 	%------------------
 	& X_i \stackrel[k_i]{m_i}{\leftrightarrows} X_1 + X_i \quad \textnormal{for } i=2,3,\dots, n~.
	%------------------
	\end{align*}
	%------------------
The ODEs for species $X_2, X_3, \dots, X_n$ are $\frac{dx_i}{dt}=0$ so, $x_i=T_i$ (with $T_i>0$) for $i=2,3,\dots,n$ are the corresponding conservation laws. We substitute these conservation laws into the ODE for~$X_1$:
	%------------------
	\begin{align}
	\label{eq:ODE-for-X1-when-only-one-non-cat-species}
	%------------------
	\frac{d x_1}{dt}|_{x_2=T_2,\dots, x_n=T_n} 
	\quad
	=
	\quad
	& (k_0 + 2 k_1)
	+ 
	(k_2T_2 + \dots + k_nT_n)
	+m_1 x_1 
	\\
	\notag
	&	
 	- (m_0+ m_2T_2 + \dots + m_nT_n) x_1
	- (\ell_0 + 2 \ell_1) x_1^2~.
	%------------------
	\end{align}
	%------------------
When at least one $k_i$ is positive and all other $k_j$'s are non-negative, the right-hand side of~\eqref{eq:ODE-for-X1-when-only-one-non-cat-species}, viewed as a polynomial in $x_1$, has a nonzero constant term and, hence, is not the zero polynomial.  Similarly, if $\ell_0$ or $\ell_1$ is positive and $\ell_0,\ell_1 \geq 0$, then the right-hand side of~\eqref{eq:ODE-for-X1-when-only-one-non-cat-species} has a nonzero coefficient of $x_1^2$ and is again a nonzero polynomial.  We conclude that if $G$ contains at least one of the reactions labeled by $k_i$ or $\ell_i$, then ${\rm cap}_{pos}(G) < \infty$, which is consistent with Lemma~\ref{lem:1-d-infinite-steady-states}.

We now consider the case when $G$ contains {\em no} reactions labeled by $k_i$ or $\ell_i$, that is, every reaction of $G$ is one of the following $n+1$ reactions:
	%------------------
	\begin{align*}
	%------------------
    &0 \stackrel{m_0}{\leftarrow} X_1 
	\stackrel{m_1}{\to} 2X_1 
	 \\
 	%------------------
 	& X_i \stackrel{m_i}{\leftarrow} X_1 + X_i \quad \textnormal{for } i=2,3,\dots, n~.
	%------------------
	\end{align*}
	%------------------

The right-hand side of the ODE for $X_1$, as in~\eqref{eq:ODE-for-X1-when-only-one-non-cat-species}, becomes $x_1(m_1 - m_0 - m_2T_2 - \dots - m_nT_n)$.  In order for this polynomial in $x_1$ to become the zero polynomial for some choice of positive rate constants of $G$ (equivalently, ${\rm cap}_{pos}(G) = \infty$), we must have $m_1>0$ and $m_j>0$ for at least one of $j=0,2,3,\dots,n$. This gives exactly
 	the reactions listed in Lemma~\ref{lem:1-d-infinite-steady-states}(2)(b). In this case, given $m_j>0$ for the reactions appearing in the network, we can always choose $T_j>0$, such that 
 the right-hand side of the ODE for $X_1$ vanishes (i.e., ${\rm cap}_{pos}(G) = \infty$).  Moreover, 
 when this right-hand side vanishes is the only situation in which there are positive steady states, and 
 an easy calculation shows that all such positive steady states are degenerate.  This concludes our analysis of networks with stoichiometric subspace spanned by the vector $(1,0,0,\dots, 0)$.

Our final case is when the stoichiometric subspace is spanned by the vector $(1,-1,0,\dots, 0)$.  In this case, the
reactions of $G$ form a subset of the following $2n+4$ reactions:
	%------------------
	\begin{align*}
	%------------------
    	&
	X_1 \stackrel[k_1]{\ell_1}{\leftrightarrows} X_2 
	\quad \quad
	 2 X_1 \stackrel[k_2]{\ell_2}{\leftrightarrows} 2X_2 
	\quad \quad 
	2 X_1 \stackrel[k_3]{m_1}{\leftrightarrows} X_1 + X_2  \stackrel[m_2]{\ell_3}{\leftrightarrows} 2 X_2 
	\\
	&
	X_1 + X_i \stackrel[k_{i+1}]{\ell_{i+1}}{\leftrightarrows} X_2 + X_i\quad  \textnormal{for } i=3,4,\dots, n~.
	%------------------
	\end{align*}
	%------------------
The conservation laws are $x_1+x_2=T_2$ and  $x_i=T_i$ for $i=3,4,\dots, n$.  The ODE for species $X_1$ is:
	\begin{align} \label{eq:final-case-proof}
	\frac{dx_1}{dt} \quad=\quad 
		&
		-(2k_2+k_3) x_1^2
		-k_1x_1
		-(k_4 x_3 + \dots + k_{n+1}x_n) x_1
		+ (m_1-m_2) x_1 x_2 
		\\
		& \notag
		\quad 
		+ (\ell_1 x_2 + 2 \ell_2 x_2^2 + \ell_3 x_2^2)
		+ (\ell_4 x_3 + \dots + \ell_{n+1} x_n) x_2~.
	\end{align}
Consider the subcase when at least one of the $\ell_i$ is positive and all other $\ell_j$'s are non-negative.  After substituting the expressions arising from the conservation laws (namely, $x_2= T_2-x_1$ and $x_i=T_i$ for $i=3,4,\dots, n$) into the right-hand side of the ODE~\eqref{eq:final-case-proof}, we obtain a polynomial in $x_1$ that has a positive constant term (see 
the second line of the right-hand side of~\eqref{eq:final-case-proof}). 
Hence, if $G$ contains at least one of the reactions labeled by $\ell_i$, then  ${\rm cap}_{pos}(G) < \infty $.

By symmetry, if $G$ has at least one of the reactions labeled by $k_i$, then again ${\rm cap}_{pos}(G) < \infty $.  Hence, if $G$ contains a reaction labeled by $\ell_i$ or $k_i$, then this subcase is consistent with the lemma.

Consider the remaining subcase, when $G$ is a subnetwork of 
$\{
2 X_1 \stackrel{m_1}{\leftarrow} X_1 + X_2  \stackrel{m_2}{\to} 2 X_2 
\}
$, and so consists of only one or two reactions.  
If $G$ has only one reaction, then 
Proposition~\ref{prop:2-or-3-rxns3species} implies that
${\rm cap}_{pos}(G) = 0 < \infty $ (which is consistent with the lemma).  

Now assume that 
$G$ has two reactions, that is, 
$G=\{
2 X_1 \stackrel{m_1}{\leftarrow} X_1 + X_2  \stackrel{m_2}{\to} 2 X_2 
\}
$.  
 If $m_1 \neq m_2$, then the ODE for $X_1$ is  $\frac{dx_1}{dt}= (m_1 - m_2) x_1 x_2$ and so there are no positive steady states. When $m_1=m_2$, the ODE for $X_1$ becomes $\frac{dx_1}{dt}=0$ and it follows that
${\rm cap}_{pos}(G) = \infty$.  Moreover, a  simple computation shows that all the positive steady states are degenerate. This concludes the proof.
\end{proof}

\begin{example_contd}[\ref{ex:generalized-degenerate-network}]	\label{ex:degenerate-1-d}
The network $\{ 0 \leftarrow A \to 2A ,~ B \leftarrow A+B\}$ is one of the networks listed in Lemma~\ref{lem:1-d-infinite-steady-states}.2(b), where $n=2$.
\end{example_contd}

%-------------------------------
\begin{lemma}[One-dimensional bimolecular networks]
\label{lem:1-dim-no-mss}
If $G$ is a one-dimensional, bimolecular network, then $G$ is \uline{not} nondegenerately multistationary.
\end{lemma}

\begin{proof}
Assume that $G$ is a one-dimensional network that is nondegenerately multistationary.  We must show that $G$ is not bimolecular.  We claim that ${\rm cap}_{pos}(G)  $ is finite.  Indeed, if ${\rm cap}_{pos}(G)  = \infty$, then Lemma~\ref{lem:1-d-infinite-steady-states} implies that all positive steady states are degenerate and so $G$ is {\em not} nondegenerately multistationary, which is a contradiction.  Hence,  ${\rm cap}_{pos}(G)  < \infty$.

The hypotheses of part~(1) of
Lemma~\ref{lem:lin-tang-zhang}
 are satisfied, that is, ${\rm cap}_{pos}(G) < \infty$, and $G$ is one-dimensional and multistationary.  Therefore, $G$ has an embedded one-species network with arrow diagram $(\leftarrow, \to)$.  Such an embedded network (e.g., $\{0 \leftarrow A,~ 2A \to 3A\}$) involves at least one complex that is {\em not} bimolecular, and so $G$ is also not bimolecular. 
\end{proof}

\begin{theorem}[Networks with up to two reversible reactions]
\label{prop:2-rev-rxn-no-mss}
If $G$ is a bimolecular reaction network that consists of one or two pairs of reversible reactions, then $G$ is not multistationary.

\end{theorem}

\begin{proof}
Assume that $G$ is bimolecular and consists of one or two pairs of reversible reactions.  
Let $p$ denote the number of pairs of reversible reactions (so, $p=1$ or $p=2$), and $\ell$ the number of linkage classes.  Let $s$ be the dimension of the stoichiometric subspace (so, $s=1$ or $s=2$).

{\bf Case 1: $p=1$}. 
The deficiency of $G$ is $\delta = 2-1-1=0$ and $G$ is weakly reversible. Hence, by the 
deficiency-zero theorem (Lemma~\ref{lem:thm:def-0}) the network is not multistationary.

{\bf Case 2: $p=s=2$}. 
If $\ell=1$, then the deficiency is $\delta = 3-1-2=0$.  If $\ell=2$, then the deficiency is $\delta=4-2-2=0$.  Therefore, for either value of $\ell$, the deficiency-zero theorem (Lemma~\ref{lem:thm:def-0})
implies that the network is not multistationary.

{\bf Case 3:} $p=2$ and $s=1$. 
$G$ is one-dimensional, bimolecular, and reversible.
So, Lemma~\ref{lem:1-d-infinite-steady-states} 
implies that ${\rm cap}_{pos}(G)  < \infty$.  Now Lemma~\ref{lem:lin-tang-zhang}(2) yields ${\rm cap}_{pos}(G) = {\rm cap}_{nondeg}(G)$,
and Lemma~\ref{lem:1-dim-no-mss} implies that ${\rm cap}_{nondeg}(G) ~\leq~ 1$. Thus, ${\rm cap}_{pos}(G) \leq 1$, or, equivalently, $G$ is non-multistationary.
\end{proof}

\section{Main results on bimolecular networks} 
\label{sec:preclude-mss-or-acr}

In this section, we establish minimal conditions for a bimolecular network to admit ACR and nondegenerate multistationarity simultaneously. 
These minimal conditions are
in terms of the numbers of species, reactions, and reactant complexes.  The main result is as follows.

\begin{theorem}[Conditions for coexistence of ACR and nondegenerate multistationarity] \label{thm:at-least-3-and-5}
Let~$G$ be a bimolecular reaction network. If there exists a vector of positive rate constants~$\kappa^*$ such that the mass-action system $(G,\kappa^*)$ has ACR and also is nondegenerately multistationary, then:
 \begin{enumerate}
 \item $G$ has at least $3$ species.
 \item $G$ has at least $3$ reactant complexes (and hence at least $3$ reactions) and at least 5 complexes (reactant and product complexes).
 \item If $G$ is full-dimensional, then $G$ has at least $5$ reactant complexes (and hence at least $5$ reactions).
 \end{enumerate}
 \end{theorem}

This section is structured as follows. 
In Subsection \ref{sec:BinaryRN},
we prove part $(1)$ of Theorem \ref{thm:at-least-3-and-5} 
(specifically, part $(1)$ follows from Proposition~\ref{prop:onespecies} and Theorem~\ref{thm:2species}). 
Theorem~\ref{thm:2species} also analyzes
two-species bimolecular networks with ACR and {\em degenerate} multistationarity.  
Additionally, we 
characterize unconditional ACR in two-species bimolecular networks that are reversible (Theorem~\ref{thm:summary-2-species}).

Subsequently, 
in Subsection \ref{sec:3-species}, 
we prove parts $(2)$ and $(3)$  of Theorem \ref{thm:at-least-3-and-5} (Theorem~\ref{thm:not_fulldim} and Proposition~\ref{prop:fulldim}).
We also consider full-dimensional, 3-species, bimolecular networks with only 4 reactant complexes.  By Theorem~\ref{thm:at-least-3-and-5}, such networks do not allow for the coexistence of ACR and nondegenerate multistationary.  Nevertheless, ACR and {\em degenerate} multistationarity is possible, and we characterize the possible sets of reactant complexes of such networks
(Proposition~\ref{prop:4rxns3species}).

\subsection{Bimolecular networks with one or two species}\label{sec:BinaryRN} % bimolecular
This subsection characterizes unconditional ACR in reversible networks with only one or two species 
(Proposition~\ref{prop:onespecies} and Theorem~\ref{thm:summary-2-species}).
Notably, our results show that such networks with unconditional ACR are {\em not} multistationary.

\begin{remark}[Reversible networks] \label{rem:reversible1}
Our interest in reversible networks comes from our prior work with Joshi~\cite{joshi-kaihnsa-nguyen-shiu-1}. In that article, our results on multistationarity in randomly generated reaction networks arise from ``lifting'' this property from the following (multistationary) motif:
	\begin{align} \label{eq:motif-from-prevalence}
	& \left\{
	B \leftrightarrows 0 \leftrightarrows A \leftrightarrows B+C~, \quad C {\leftrightarrows} 2C
		\right\}~.
	\end{align}
The question arises, {\em Are there multistationary motifs with fewer species, reactions, or complexes than the one in~\eqref{eq:motif-from-prevalence}?  }
Discovering more motifs might aid in analyzing the prevalence of multistationarity in random reaction networks generated by stochastic models besides the one in~\cite{joshi-kaihnsa-nguyen-shiu-1}.
\end{remark}

\subsubsection{Networks with one species} \label{sec:1-species}

When there is only one species, say $X_1$, 
and the network is bimolecular, 
there are only $3$ possible complexes: $0, X_1, 2X_1$. 
Hence, every such network is a subnetwork of the following network: 
	\begin{align} \label{eq:1-species-net}
	G_{X_1} ~=~	\{0\leftrightarrows X_1\leftrightarrows 2X_1 \leftrightarrows 0\}.
	\end{align}

Therefore, the possible reversible networks, besides $G_{X_1}$ itself, are listed here:
\begin{small}
\begin{align} \label{eq:1-species-bimolecular-rev-nets}
\{0\leftrightarrows X_1\},~
\{X_1\leftrightarrows 2X_1\},~
\{0\leftrightarrows 2X_1\},~ \{0\leftrightarrows X_1\leftrightarrows 2X_1\},~
\{X_1\leftrightarrows 0\leftrightarrows 2X_1\},~
\{0\leftrightarrows 2X_1\leftrightarrows X_1\}~. 
\end{align}
\end{small}
	\begin{proposition}\label{prop:onespecies}
		Every bimolecular network in only one species 
		is \uline{not} nondegenerately multistationary.  
		Every reversible, bimolecular network in only one species 
		has unconditional ACR. 
	\end{proposition}
	\begin{proof} 
		Let $G$ be a  bimolecular network with only one species.  Then $G$ is a subnetwork of $G_{X_1}$, in~\eqref{eq:1-species-net}, and the first part of the proposition now follows readily from Lemmas~\ref{lem:lin-tang-zhang}--\ref{lem:1-d-infinite-steady-states}.
		
		Next, assume $G$ is a reversible, bimolecular network with only one species.  Then $G$ is either the network $G_{X_1}$ or one of the networks listed in~\eqref{eq:1-species-bimolecular-rev-nets}.
		Each of these networks is weakly reversible and satisfies the conditions of either the 
		deficiency-zero  
		or deficiency-one theorem 
		(Lemmas~\ref{lem:thm:def-0}--\ref{lem:thm:def-1}). 
		Thus, for every choice of positive rate constants $\kappa$, the mass-action system $(G, \kappa)$ has a unique positive steady state.  Hence, 
		$G$ has unconditional ACR. 
	\end{proof}
\subsubsection{Reversible networks with two species} \label{sec:2-species}

We now consider reversible, bimolecular networks with two species. 
Among such networks, the ones with unconditional ACR are characterized in the following result, which is the main result of this subsection.

%-------------------------------
% MAIN THEOREM OF SUBSECTION
%-------------------------------
\begin{theorem}[Unconditional ACR in reversible, 2-species networks] \label{thm:summary-2-species}
Let $G$ be a reversible, bimolecular reaction network with exactly two species (and at least one reaction). 
\begin{enumerate}
	\item If $G$ is full-dimensional, then the following are equivalent:
		\begin{enumerate}[(a)]
		\item $G$ has unconditional ACR;
		\item $G$ is not multistationary.
		\end{enumerate}
	\item If $G$ is one-dimensional, then the following are equivalent:
		\begin{enumerate}[(a)]
		\item $G$ has unconditional ACR;
		\item Up to relabeling species, $G$ is the (non-multistationary) network $\{X_2 \leftrightarrows X_1+X_2\}$.
		\end{enumerate}
\end{enumerate}
\end{theorem}

Theorem~\ref{thm:summary-2-species} encompasses Propositions~\ref{prop:2-species-no-conservation laws} and~\ref{prop:2-species-cons-law} below.  
%Next, we use Propositions~\ref{lem:when-ODE-is-0} and~\ref{lem:when-evaluated-ODE-is-0} to 
%prove the following result, which is part~(1) of Theorem~\ref{thm:summary-2-species}.

\begin{proposition} \label{prop:2-species-no-conservation laws}
Let $G$ be a full-dimensional, reversible, bimolecular reaction network with exactly two species.
Then the following are equivalent:
		\begin{enumerate}[(a)]
		\item $G$ has unconditional ACR;
		\item $G$ is not multistationary.
		\end{enumerate}
\end{proposition}

\begin{proof}  Let $G$ 
be a full-dimensional, reversible, bimolecular network with exactly $2$ species.

We first prove $(b) \Rightarrow (a)$. Assume that $G$ is non-multistationary, and let 
$\kappa^*$ be a choice of positive rate constants.  Then, the mass-action system $(G,\kappa^*)$ admits {\em at most} one positive steady state $(x^*_1, x^*_2)$ (here the assumption that $G$ is full-dimensional is used).  
However, the fact that $G$ is reversible guarantees {\em at least} one positive steady state (Remark~\ref{rem:reversible}).
Hence,  $(G, \kappa^*)$ has a unique positive steady state $(x^*_1, x^*_2)$ and therefore 
has ACR in both species with ACR-values $x_1^*$ and $x_2^*$, respectively.  So, $G$ has unconditional ACR.

Next, we prove $(a) \Rightarrow (b)$.   Assume that $G$ has unconditional ACR.  Let $\kappa^*$ be a choice of positive rate constants.  By relabeling species, if necessary, we may assume that the system $(G, \kappa^*)$ has ACR in species $X_1$ with some ACR-value $\alpha>0$.  Every positive steady state of $(G, \kappa^*)$, therefore,  has the form $(\alpha, x_2^*)$, where $x^*_2 \in \mathbb{R}_{>0}$.  We must show that there is at most one such steady state.

Write the mass-action ODEs of $(G, \kappa^*)$ as $\frac{dx_1}{dt} = f_1$ and $\frac{dx_2}{dt}=f_2$. Consider the univariate polynomial $f_2|_{x_1 = \alpha} \in \mathbb{R}[x_2]$.  We claim that this polynomial is not the zero polynomial.  To check this claim,  assume for contradiction that $f_2|_{x_1 = \alpha}$ is zero. As $G$ is reversible,
Remark~\ref{rmk:lemma} (which relies on
Propositions~\ref{lem:when-ODE-is-0}--\ref{lem:when-evaluated-ODE-is-0})
implies that $X_2$ is a catalyst-only species of every reaction of $G$.  We conclude that $G$ is not full-dimensional, which is a contradiction.
Having shown that the univariate polynomial 
$f_2|_{x_1 = \alpha}$ is nonzero, we now use Lemma~\ref{lem:bimol-Descartes} to conclude that $(G, \kappa^*)$ 
has at most one positive steady state of the form $(\alpha, x_2^*)$.
\end{proof}

Proposition~\ref{prop:2-species-no-conservation laws} fails for networks that are not reversible.  Indeed, a network without positive steady states (such as $\{ 0 \to A,~ 0 \to B\}$) is not multistationary and also lacks unconditional ACR.

We end this subsection by considering two-species networks that are one-dimensional. 
Up to relabeling species, each such network is a subnetwork of exactly one of the following networks~$G_i$:
\begin{align}
\label{eq:4-networks}
G_1~&:=~ \{ 0 \leftrightarrows X_1+X_2 \}  \notag  \\
G_2~&:=~ \{X_1\leftrightarrows 2X_2\} \\
G_3~&:=~
	\{ 0 \leftrightarrows X_1 \leftrightarrows 2X_1 \leftrightarrows 0~, 
	~ X_2 \leftrightarrows X_1+X_2\}
 \notag \\
G_4 ~&:=~ \{ 2X_1 \leftrightarrows X_1+X_2 \leftrightarrows 2X_2 \leftrightarrows 2X_1 ~, ~ X_1 \leftrightarrows X_2 \}
\notag 
\end{align}

The next result, which is part~(2) of Theorem~\ref{thm:summary-2-species}, 
states that among the reversible subnetworks of the networks $G_i$ listed in~\eqref{eq:4-networks}, 
only one has unconditional ACR (namely, $\{X_2 \leftrightarrows X_1+X_2\}$).

\begin{proposition} \label{prop:2-species-cons-law} 
Let $G$ be a one-dimensional, reversible, bimolecular reaction network with exactly two species. Then the following are equivalent:
		\begin{enumerate}[(a)]
		\item $G$ has unconditional ACR;
		\item Up to relabeling species, $G$ is the (non-multistationary) network $\{X_2 \leftrightarrows X_1+X_2\}$.
		\end{enumerate}
\end{proposition}

\begin{proof}
Let $G$ be a two-species, one-dimensional, reversible, bimolecular reaction network.    
From the list~\eqref{eq:4-networks}, we know that $G$ is a subnetwork of one of $G_1$, $G_2$, $G_3$, and $G_4$.

Assume $G$ is a subnetwork of $G_1$, $G_2$, or $G_4$.  Then, 
$G \neq \{X_2 \leftrightarrows X_1+X_2\}$ and $G$ is not a subnetwork of $\{0 \leftrightarrows X_1 \leftrightarrows  2X_1 \leftrightarrows  0\}$.  So, it suffices to show $G$ does not have unconditional ACR.  

In networks $G_1$, $G_2,$ and $G_4$, the reactant and product complexes of every reaction differ in both species $X_1$ and $X_2$.  Also, all reactions in $G$ are reversible, so every complex of $G$ is a reactant complex. We conclude that $G$ has two reactant complexes that differ in both species, and hence, Lemma~\ref{lem:1-dim-reactants} implies that $G$ does not have unconditional ACR. 

We now consider the remaining case, when $G$ is a subnetwork of $G_3$.  
We write $G_3 = N_1 \cup N_2$, where $N_1:=\{0 \leftrightarrows X_1 \leftrightarrows  2X_1 \leftrightarrows  0 \}$ and $N_2:=\{X_2 \leftrightarrows X_1+X_2\}$.  
If $G=N_2$, the mass-action ODEs are $d{x}_1/dt =\kappa_1x_2-\kappa_2x_1x_2$ and $d{x}_2/dt =0$, and so $G$ has unconditional ACR in species $X_1$ with ACR-value $\tfrac{\kappa_1}{\kappa_2}$.  
% Subcase - G1
If $G$ is a subnetwork of $N_1$, then $G$ has only one species (recall that every species of a network must take part in at least one reaction), which is a contradiction.

Our final subcase is when $G$ contains reactions from both $N_1$ and $N_2$.   
Then, from $N_2$, the complex $X_2$ %and $A+B$ are reactant complexes 
is a reactant complex of $G$. 
Similarly, from $N_1$, at least one of $X_1$ and $2X_1$ is a reactant complex of $G$. 
Hence, $G$ contains two reactant complexes that differ in both species, $X_1$ and $X_2$. 
Therefore, Lemma~\ref{lem:1-dim-reactants} implies that  $G$ does not have unconditional ACR.

Finally, the fact that the network $\{X_2 \leftrightarrows X_1+X_2\}$ is
non-multistationary follows easily from the deficiency-zero theorem (Lemma~\ref{lem:thm:def-0}). 
\end{proof}

\subsubsection{Irreversible networks with two species}
In~\cite{MST}, the following network was called a ``degenerate-ACR network,'' because it has unconditional ACR and yet every positive steady state is degenerate:
\begin{align} \label{eq:degenerate-ACR-network}
	\{ A+B \to B,~A \to 2A \}~.
\end{align}
This degeneracy arises from the fact that a single (one-dimensional) stoichiometric compatibility class consists entirely of steady states~\cite[Example~2.12]{MST}.
The main result of this subsection, 
Theorem~\ref{thm:2species} below,
shows that only one additional
two-species network exhibits both ACR and multistationarity for a nonzero-measure set of rate constants; this network is 
obtained by adding to~\eqref{eq:degenerate-ACR-network}  the reaction $A \to 0$.  Both networks, therefore, are one-dimensional, two-species networks.

To prove Theorem~\ref{thm:2species}, we need the following lemma, which concerns the network in~\eqref{eq:degenerate-ACR-network} (and others as well).

\begin{lemma} \label{lem:degenerate-networks-2-species}
Let $G$ be a subnetwork of the network $\{X_1+X_2 \to X_2,~0 \leftarrow X_1 \to 2X_1\}$.  Then:
	\begin{enumerate}
	\item Every positive steady state (of \uline{every} mass-action system defined by $G$) is degenerate. 
	\item Let $\Sigma$ denote the 
	set of vectors of positive rate constants $\kappa$ 
for which the mass-action system $(G, \kappa)$ both has ACR and is multistationary. If $\Sigma$ has nonzero measure, then $G$ is one of the following networks:  
	$\{X_1+X_2 \to X_2,~X_1 \to 2X_1\}$ and 
	$\{X_1+X_2 \to X_2,~0 \leftarrow X_1 \to 2X_1\}$.
	\end{enumerate}

\end{lemma}

\begin{proof}
This result is straightforward to check by hand, so we only outline the steps, as follows.  
Assume $G$ is a subnetwork of $\{X_1+X_2 \overset{k}\to X_2,~0 \overset{\ell}\leftarrow X_1 \overset{m}\to 2X_1\}$.  If $G$ admits a positive steady state, $G$ must 
contain the reaction $X_1  \overset{m}\to 2X_1$.  Hence, there are three subnetworks to consider:
	\begin{enumerate}
	\item If $G= \{0 \overset{\ell}\leftarrow X_1  \overset{m}\to 2X_1\}$, then $\Sigma$ is empty.
	\item If $G= \{X_1+X_2 \overset{k}\to X_2,~ X_1  \overset{m}\to 2X_1\}$, then $\Sigma = \{ (k,m) \in \mathbb{R}^2_{>0}  \}$.
	\item If $G= \{X_1+X_2 \overset{k}\to X_2,~0 \overset{\ell}\leftarrow X_1  \overset{m}\to 2X_1\}$, then $\Sigma = \{ (k,\ell,m) \in \mathbb{R}^3_{>0} \mid m > \ell \}$.
	\end{enumerate}
In cases (2) and (3), the set $\Sigma$ has nonzero measure.  Finally, for all three of these networks, every positive steady state is degenerate
(some of these 
networks are also covered by Lemma~\ref{lem:1-d-infinite-steady-states}).
\end{proof}

\begin{theorem}\label{thm:2species} 
Let $G$ be a bimolecular reaction network with exactly two species, $X_1$ and $X_2$. 
Let $\Sigma$ denote the set of vectors of positive rate constants $\kappa$ 
for which the mass-action system $(G, \kappa)$ both has ACR in species $X_2$ and is multistationary.  Then:
\begin{enumerate}
	\item For every $\kappa^* \in \Sigma$, every positive steady state of $(G, \kappa^*)$ is degenerate.
	\item If $\Sigma$ has nonzero measure, then  
	$G$ is one of the following networks:  
	$\{X_1+X_2 \to X_2,~X_1 \to 2X_1\}$ and 
	$\{X_1+X_2 \to X_2,~0 \leftarrow X_1 \to 2X_1\}$.
\end{enumerate}
\end{theorem}

\begin{proof}  
Assume that $G$ is bimolecular and has exactly two species.  
If $\Sigma$ is empty (for instance, if $G$ has no reactions), then there is nothing to prove.	
Accordingly, assume that $\Sigma$ is nonempty (and in particular $G$ has at least one reaction). 

We first claim that $G$ has a reaction in which $X_1$ is a non-catalyst-only species.  To prove this claim, assume for contradiction that $X_1$ is a catalyst-only species.  Then the stoichiometric compatibility classes are defined by the equations $x_1=T$, for $T>0$ (we are also using the fact that $G$ has at least one reaction).  But this does not allow for multistationarity and ACR in $X_2$ to coexist, because two positive steady states in the same compatibility class would have the form $(T,y)$ and $(T,z)$, with $y\neq z$, which contradicts the assumption of ACR in $X_2$.
So, the claim holds.

For an arbitrary vector $\kappa$ of positive rate constants, let $f_{\kappa,1}$ and $f_{\kappa,2}$ denote the right-hand sides (for species $X_1$ and $X_2$, respectively) of the mass-action ODE system of $(G,\kappa)$. 
Consider the following partition of $\Sigma$:
% PARTITION
\begin{align*} 
	\Sigma ~=~ \left( \Sigma \cap \{\kappa \mid f_{\kappa,1}= 0\} \right)
		\cup 
		\left( \Sigma \cap \{\kappa \mid f_{\kappa,1} \neq 0\} \right)
	~=:~ \Sigma_0 \cup \Sigma_1~.
\end{align*}
 
By construction, $\Sigma_0\cap \Sigma_1=\emptyset$. We first analyze $\Sigma_0$. 
If $\Sigma_0$ is empty, then skip ahead to our analysis of $\Sigma_1$.  Accordingly, assume $\Sigma_0$ is nonempty, and let $\kappa^* \in \Sigma_0$.  We must show that every positive steady state of $(G,\kappa^*)$ is degenerate.

We claim that $G$ is two-dimensional (assuming that $\Sigma_0$ is nonempty).
We prove this claim as follows.  We saw that $G$ contains 
 a reaction in which $X_1$ is a non-catalyst-only species, so 
  Proposition~\ref{lem:when-ODE-is-0}
implies that for $j=0$ or $j=2$ (or both, where we are using Notation~\ref{notation:dummy-variable}) 
our network 
$G$ contains the reaction 
$X_1+X_j \to 2X_1$ 
and at least one reaction of the form $X_1+X_j \to \star$, where $\star$ is a complex not involving $X_1$.  
Consider the subcase $j=0$.  If some $\star$ involves $X_2$, then $G$ contains $X_1 \to 2X_1$ and $X_1 \to \star$, which yield linearly independent reaction vectors and so $G$ is two-dimensional.  If none of the complexes $\star$ involve $X_2$, then 
$G$ must contain additional reactions in which $X_2$ is not a catalyst-only species (to avoid $f_2=0$), and so again $G$ is two-dimensional.  
The subcase $j=2$ is similar.

Next, as $G$ is two-dimensional and $f_{\kappa^*,1}= 0$, Corollary~\ref{cor:degeneracy} implies that every positive steady state of $(G,\kappa^*)$ is degenerate, as desired.
Additionally, as $X_1$ is a non-catalyst-only species and (for all $\kappa \in \Sigma_0$) $f_{\kappa,1}=0$, 
Proposition~\ref{lem:when-ODE-is-0} implies that there is a nontrivial linear relation that every $\kappa \in \Sigma_0$ satisfies. Hence, $\Sigma_0$ has zero measure.   
 
To complete the proof, it suffices to show the following about the set $\Sigma_1$: (1) 
For every $\kappa^* \in \Sigma_1$, every positive steady state of $(G, \kappa^*)$ is degenerate; and (2) If $\Sigma_1$ has nonzero measure, then 
	$G=\{X_1+X_2 \to X_2,~X_1 \to 2X_1\}$ or
	$G=\{X_1+X_2 \to X_2,~0 \leftarrow X_1 \to 2X_1\}$.

Assume $\Sigma_1$ is nonempty (otherwise, there is nothing to prove).
We introduce the following notation: for $\widetilde \kappa \in \Sigma_1$, let $\beta(\widetilde \kappa)$ denote the ACR-value for $X_2$.  

We now claim the following:
For every $\widetilde \kappa \in \Sigma_1$, 
	the univariate polynomial $f_{\widetilde \kappa,_1}|_{x_2=\beta(\widetilde \kappa) }$ is the zero polynomial.  	
	To verify this claim, we first note that
	$f_{\widetilde \kappa,_1}|_{x_2=\beta(\widetilde \kappa) }$ has at least two positive roots (as $(G, \widetilde \kappa)$ is multistationarity), 
	so the polynomial $f_{\widetilde \kappa,_1}|_{x_2=\beta(\widetilde \kappa) }$, if nonzero, must have at least two sign changes (by Descartes' rule of signs).
	 However, by Lemma~\ref{lem:bimol-Descartes}, the polynomial $f_{\widetilde \kappa,_1}|_{x_2=\beta(\widetilde \kappa) }$ has at most one sign change, and so the claim holds.  

We now know that for every $\kappa^* \in \Sigma_1$, we have $f_{\kappa^*,_1} \neq 0$, but $f_{ \kappa^*,_1}|_{x_2=\beta(\widetilde \kappa^*) }=0$.  Hence, 
	$G$ has at least one reaction in which $X_1$ is a non-catalyst-only reaction and (by Proposition~\ref{lem:when-evaluated-ODE-is-0}) every such reaction must be one of the $8$ reactions displayed here:
	% ---------------
	% POSSIBLE REACTIONS	
	% ---------------
	\begin{align} \label{eq:only-possible-rxns}
	0 \overset{ \kappa_{4,1} } \longleftarrow X_1 \overset{\kappa_1} \longrightarrow  2X_1  
	\overset{\kappa_2} \longleftarrow X_1+X_2 \overset{\kappa_{3,1} } \longrightarrow 0 ~, \quad 
	X_2 \overset{\kappa_{4,2}} \longleftarrow X_1 \overset{\kappa_{4,3}} \longrightarrow 2X_2~,
	\quad  
	X_2 \overset{\kappa_{3,2}} \longleftarrow X_1+X_2 \overset{\kappa_{3,3}} \longrightarrow 2X_2	
	\end{align}

For every $\kappa^* \in \Sigma_1$, Proposition~\ref{lem:when-evaluated-ODE-is-0} yields
the following ACR-value formula:

		\begin{align} \label{eq:beta-value}
		\beta(\kappa^*) ~=~ \frac{ \kappa^*_{4 \bullet } - \kappa^*_1}{\kappa^*_2 - \kappa^*_{3 \bullet} }~, 
	\end{align}
	where 
	$\kappa^*_{3 \bullet } := \kappa^*_{3, 1 } + \kappa^*_{3, 2 }  + \kappa^*_{3, 3 } $
	and 
	$\kappa^*_{4 \bullet }:=\kappa^*_{4, 1 } + \kappa^*_{4, 2 }  + \kappa^*_{4, 3 } $.
	For reactions in~\eqref{eq:only-possible-rxns} that are not in~$G$, the corresponding rate constants,  $\kappa^*_i$ or $\kappa^*_{ij}$, are set to 0. 

Next, the possible reactions in which $X_1$ is a catalyst-only species are as follows:
	\begin{align} \label{eq:only-possible-cat-rxns}
	0 \overset{\kappa_5} {\underset{\kappa_6} \rightleftarrows}  X_2
		\overset{\kappa_7} {\underset{\kappa_8} \rightleftarrows} 2X_2
		\overset{\kappa_9} {\underset{\kappa_{10}} \rightleftarrows}  0~, \quad \quad 
	X_1 \overset{\kappa_{11}} {\underset{\kappa_{12}} \rightleftarrows}  X_1+ X_2
	\end{align}

We proceed by considering three subcases, based in part on whether $f_{\kappa,2}$ (which is a polynomial in the unknowns $x_1$, $x_2$, and $\kappa$) is zero: 
	 \begin{itemize}
	 \item[(a)] $f_{\kappa,2}=0$, and $X_2$ is a catalyst-only species in every reaction of $G$,
	 \item[(b)] $f_{\kappa,2}=0$, and $X_2$ is a non-catalyst-only species in some reaction of $G$, or
	 \item[(c)] $f_{\kappa,2} \neq 0$.
	 \end{itemize}

	 We first consider subcase (a).  
	 By inspecting reactions in~(\ref{eq:only-possible-rxns}) and (\ref{eq:only-possible-cat-rxns}), we conclude that $G$ must be a subnetwork of  $\{X_1+X_2 \to X_2,~0 \leftarrow X_1 \to 2X_1\}$.  This subcase is done by Lemma~\ref{lem:degenerate-networks-2-species}.

Next, we examine subcase (b). 
Let $G_1:=\{X_1+X_2 \to X_2,~0 \leftarrow X_1 \to 2X_1\}$, $G_2:= \{0\leftarrow X_2 \to 2X_2\}$, and 
$G_3:=\{0\leftarrow X_1+X_2\to 2X_2,~ X_1 \leftarrow X_1+ X_2 \to 2X_1 \}$.
By Proposition~\ref{lem:when-ODE-is-0} 
(and by inspecting reactions in~(\ref{eq:only-possible-rxns}) and (\ref{eq:only-possible-cat-rxns})), $G$ must be a subnetwork of $G_1\cup G_2 \cup G_3$ with at least one reaction in $G_2 \cup G_3$.
Moreover, there is a nontrivial linear relation in the rate constants that holds for all $\kappa \in \Sigma_1$.
It follows that $\Sigma_1$ is contained in the hyperplane defined by this linear relation and hence has zero measure. 

Let $\kappa^* \in \Sigma_1$. 
By examining $G_1 \cup G_2 \cup G_3$, we see that the possible reactants of $G$ are $X_1,~X_2,~X_1+X_2$.  Next, $G$ has at least $2$ reactants (as otherwise, Proposition~\ref{prop:2-or-3-rxns3species} would imply that $G$ admits no positive steady states).
Hence, by inspection, $G$ either is full-dimensional or is a subnetwork of $\{ 0 \leftarrow X_2 \to 2X_2,~ X_1+X_2 \to X_1\}$, which we already saw in Example~\ref{ex:generalized-degenerate-network} (where $A=X_2$ and $B=X_1$) has ACR in $X_1$ but not in $X_2$ (and the analysis of its subnetworks is similar).   Hence, $G$ is full-dimensional, and so
Corollary~\ref{cor:degeneracy} (and the fact that $f_{\kappa^*,2}=0$) implies that every positive steady state of $(G,\kappa^*)$ is degenerate.

Consider subcase (c).  
Let $\kappa^* \in \Sigma_1$ (so, in particular, $f_{\kappa^*,2} \neq 0$).  We claim that $f_{\kappa^*,2}|_{x_2=\beta(\kappa^*)} = 0$.  To see this, observe that, in the reactions~(\ref{eq:only-possible-rxns}) and (\ref{eq:only-possible-cat-rxns}), the complex $2X_1$ appears only as a product, never as a reactant.  
	 Hence, $f_{\kappa^*,2}|_{x_2=\beta(\kappa^*)} $ (which is a univariate polynomial in $x_1$) has degree at most 1. 
	However, the fact that $(G,\kappa^*)$ is multistationary implies that $f_{\kappa^*,2}|_{x_2=\beta(\kappa^*)} $ has two or more positive roots.  
	Hence, $f_{\kappa^*,2}|_{x_2=\beta(\kappa^*)}$ is the zero polynomial.

	Now we show that every positive steady state of $(G,\kappa^*)$ is degenerate.  Such a steady state has the form $(p,\beta)$, and we also know that 
	$f_{\kappa^*,1}|_{x_2=\beta(\kappa^*)} = f_{\kappa^*,2}|_{x_2=\beta(\kappa^*)}=0$.  
	 Hence, $(x_2-\beta(\kappa^*))$ divides both $f_{\kappa^*,1}$ and $f_{\kappa^*,2}$. Consequently, the derivatives of $f_{\kappa^*,1}$ and $f_{\kappa^*,2}$ with respect to $x_1$ at $(p,\beta)$ are both zero. It follows that the first column of the $2 \times 2$ Jacobian matrix, when evaluated at $(p,\beta)$, is the zero column.  Hence, if the stoichiometric subspace of $G$, which we denote by $S$, is two-dimensional, then $(p,\beta)$ is degenerate.

	We now assume  
	$\dim(S)=1$ (and aim to reach a contradiction).
	 Recall that $G$ contains at least one reaction from those in~\eqref{eq:only-possible-rxns}, so in order for $\dim(S)=1$ it must be that $G$ contains no reaction from~\eqref{eq:only-possible-cat-rxns}.  Hence, from the expression for $f_2$ (which we know is not zero), in~\eqref{eq:f_2}, the only possible reactions  in $G$ are the ones labeled by
	 $\kappa_2 ,  \kappa_{3,3}, \kappa_{4,2},\kappa_{4,3}$.  Hence, the one-dimensional network $G$ is either the network $\{ X_1 \overset{\kappa_{4,3}}\longrightarrow 2X_2 \}$ or a subnetwork of $\{2X_1  
	\overset{\kappa_2} \longleftarrow X_1+X_2 \overset{\kappa_{3,3} } \longrightarrow 2X_2, ~ X_1 \overset{\kappa_{4,2} } \longrightarrow X_2 \}$. 
	Now it is straightforward to check that $G$ is not multistationary, which is a contradiction.

	To complete the proof, it suffices to show that, in subcase~(c), the set $\Sigma_1$ has measure zero.
	Accordingly, let $\kappa \in \Sigma_1$.  
	As noted earlier, the ACR-value of $X_2$ in $(G,\kappa)$ is $\beta(\kappa) = \frac{ \kappa_{4 \bullet } - \kappa_1}{\kappa_2 - \kappa_{3 \bullet} }$.	
	From~\eqref{eq:only-possible-rxns} and~\eqref{eq:only-possible-cat-rxns}, 
	the right-hand side of the mass-action ODE for $(G,\kappa)$
	has the following form (with rate constants set to $0$ for reactions not in $G$):
	\begin{small}
	\begin{align} \label{eq:f_2}
		f_{\kappa,2} ~=~ 
			( \kappa_{3,3} - \kappa_2 - \kappa_{12}) x_1 x_2
			+ (\kappa_{11} + \kappa_{4,2} + 2 \kappa_{4,3}) x_1
			- (\kappa_8 + 2 \kappa_9) x_2^2
			+ (\kappa_7 - \kappa_6) x_2 
			+ (\kappa_5 + 2 \kappa_{10})~.
	\end{align}
	\end{small}
	\noindent By assumption, at least one of the rate constants 
%	(the $\kappa^*_i$ and $\kappa^*_{i,j}$) 
	(the $\kappa_i$ and $\kappa_{i,j}$) 
	in~\eqref{eq:f_2} is nonzero. 
	By our earlier arguments,  
	at the beginning of subcase (c), we conclude that $f_{\kappa,2}|_{x_2=\beta(\kappa)} = 0$.
	Hence, the linear and constant terms of 
	$f_{\kappa,2}|_{x_2=\beta(\kappa)} $ are both $0$, which, using~\eqref{eq:f_2}, translates as follows: 
	\begin{align} \label{eq:coeffs=0}
	( \kappa_{3,3} - \kappa_2 - \kappa_{12}) \frac{ \kappa_{4 \bullet } - \kappa_1}{\kappa_2 - \kappa_{3 \bullet} }
			+ (\kappa_{11} + \kappa_{4,2} + 2 \kappa_{4,3}) 
			~&=~ 0 
			\quad 
			{\rm and} \\
			%\quad 	
			\notag
			- (\kappa_8 + 2 \kappa_9) \left( \frac{ \kappa_{4 \bullet } - \kappa_1}{\kappa_2 - \kappa_{3 \bullet} } \right)^2
			+ (\kappa_7 - \kappa_6) \frac{ \kappa_{4 \bullet } - \kappa_1}{\kappa_2 - \kappa_{3 \bullet} } 
			+ (\kappa_5 + 2 \kappa_{10}) ~&=~0~.	
	\end{align}
It follows that $\Sigma_1$ is constrained by the equations~\eqref{eq:coeffs=0}, at least one of which is nontrivial.  
 Hence, $\Sigma_1$ is contained in a hypersurface and so has measure zero.
\end{proof}

\subsection{Bimolecular networks with at least three species} \label{sec:3-species}
In the previous subsection, we showed that a bimolecular network must have at least $3$ species in order for ACR and nondegenerate multistationarity to coexist. Consequently, this subsection focuses on bimolecular networks with 
at least $3$ species. 
We prove that the coexistence of ACR and nondegenerate multistationarity requires a minimum of $3$ reactant complexes and a minimum of $5$ complexes (Theorem~\ref{thm:not_fulldim}).  
The remainder of this subsection focuses on a family of networks with $n$ species and $n$ reactants,
for which ACR and nondegenerate multistationarity coexist
(Section~\ref{sec:cons-law}), and then analyzes full-dimensional networks with $3$ species (Section~\ref{sec:full-d-3-species}).

\begin{theorem}[Minimum number of complexes] \label{thm:not_fulldim}
Let $G$ be a bimolecular reaction network with at least 3 species. If there exists a vector of positive rate constants $\kappa^*$ such that $(G,\kappa^*)$  has ACR and is nondegenerately multistationary, then: 
\begin{enumerate}
\item $G$ has at least 3 reactant complexes (and hence, at least 3 reactions), and
\item $G$ has at least 5 complexes (reactant and product complexes).
\end{enumerate}
\end{theorem}

\begin{proof} 
We first prove part (1). Let $(G,\kappa^*)$ be as in the statement of the theorem and let $n$ denote the number of species, where $n \geq 3$.  
By relabeling species, if needed, we may assume that $(G,\kappa^*)$ has ACR in species $X_1$. 
Let $f_1,f_2,\dots, f_n$ denote the right-hand sides of the mass-action ODEs of $(G,\kappa^*)$.
As $(G,\kappa^*)$ has ACR, we know that at least one of the right-hand sides is nonzero.  Let $f_i$ denote one of these nonzero polynomials.

Assume for contradiction that $G$ has only $1$ or $2$ reactant complexes. Since $G$ admits a nondegenerate positive steady state,
 Proposition~\ref{prop:2-or-3-rxns3species} implies that $G$ is not full-dimensional and 
 $G$ has exactly $2$ reactant complexes. 

We claim that all the right-hand sides $f_{\ell}$ are scalar multiples of each other. 
More precisely, we claim that for all $j \in \{1,2,\dots,n\} \smallsetminus \{i\}$, there exists $c_j \in \mathbb{R}$ such that $f_j=c_j f_i$.
Indeed, each $f_j$ has at most two monomials (because $G$ has exactly two reactant complexes), so if $f_j$ is not a constant multiple of $f_i$, then some $\mathbb{R}$-linear combination of $f_i$ and $f_j$ is a monomial and hence $(G,\kappa^*)$ has no positive steady state (which is a contradiction).

Thus the positive steady states of $(G,\kappa^*)$ are precisely the positive roots of $f_i=0$ and the linear equations given by the conservation laws. 
Since $X_1$ is the ACR species and $G$ is bimolecular, we must have $f_i = (\alpha - x_1) (\beta_0 + \beta_1 x_1 + \dots + \beta_n x_n)$, 
where $\alpha$ is the (positive) ACR-value and $\beta_j \in \mathbb{R}$ for all $j=0,1,\dots,n$. 

We consider several cases, based on how many of the coefficients $\beta_2,\beta_3, \dots, \beta_n$ are nonzero.
We begin by considering the case when $\beta_2=\beta_3=\dots = \beta_n=0$.
In this case, 
$f_i$ is a (nonzero) polynomial in $x_1$ only, and so has the form $f_i=\gamma_1 x_1^{m_1} + \gamma_2 x_1^{m_2} $, where $\gamma_1,\gamma_2 \in \mathbb{R}$ and $0 \leq m_1 < m_2 \leq 2$.  As $(G,\kappa^*)$ has a positive steady state, we conclude that $\gamma_1$ and $\gamma_2$ are nonzero and have opposite signs.  Now Lemma~\ref{lem:hungarian} implies that $i=1$ (so, $f_1=f_i \neq 0$) and $f_2=f_3=\dots=f_n=0$.  In fact, Lemma~\ref{lem:hungarian} implies that $X_2, \dots, X_n$ are catalyst-only species of $G$ (equivalently, the mass-action ODE right-hand sides for $X_2,\dots, X_n$ are zero for all choices of positive rate constants).
Such a system is {\em not} multistatationary, which is a contradiction.

Now consider the case when two or more of the $\beta_2,\beta_3, \dots, \beta_n$ are nonzero.
In this case, 
there exist distinct $ j_1,j_2$ (where $2 \leq j_1,j_2 \leq n$) with %$j_1\neq 0,1$ and 
$\beta_{j_1},\beta_{j_2}\neq 0$. 
Then $f_i$ contains the monomials $x_{j_1},~x_{j_2},~x_1x_{j_1},~x_1x_{j_2}$ which contradicts the fact that $G$ has exactly two reactant complexes. %, where $x_0:=1$. 

The final case is when exactly one of the $\beta_2,\beta_3, \dots, \beta_n$ is nonzero. Relabel the species, if needed, so that $\beta_2\neq 0$.
In this case, the two reactant complexes of $G$ involve only species $X_1$ and $X_2$. By using Lemma~\ref{lem:hungarian} again, much like we did for the prior case, we conclude that $X_3 ,\dots, X_n$ are catalyst-only species of $G$ 
 and so $(G,\kappa^*)$ is effectively the mass-action system of a (bimolecular) network with only two species, $X_1$ and $X_2$. Now it follows from Theorem~\ref{thm:2species} that $(G,\kappa^*)$ is {\em not} nondegenerately multistationarity, which contradicts our assumption. This completes part~(1).

We prove part (2). Assume for contradiction that $G$ has at most 4 complexes. By Proposition~\ref{prop:1-d-no-coexistence}, the dimension of the stoichiometric subspace of $G$ must be at least $2$. 
So, the deficiency of $G$ satisfies
\[
\delta ~=~ m - \ell - \dim(S) ~\leq~ 4 - 1 - 2 ~=~1~.
\]
Hence, the deficiency of $G$ is $0$ or $1$, and the latter requires $G$ to have exactly one linkage class. Now  Lemmas \ref{lem:thm:def-0}--\ref{lem:thm:def-1} imply that $G$ is not multistationary, which is a contradiction.
\end{proof}

Theorem \ref{thm:not_fulldim} gives a lower bound on the number of reactant complexes and the number of all complexes (reactants and products), and the next example shows that these bounds are tight.  The example also shows the tightness of the lower bounds on the number of species and the dimension of the network (from Theorem~\ref{thm:summary-theorem}). 

\begin{example}\label{example:min3} 
Consider the following bimolecular network with 3 species and 3 reactant complexes and 5 complexes:
\[
G~=~
\{X_1+X_2 \xrightarrow{\k_1} 2X_3, \quad X_3\xrightarrow{\k_2} X_1 , \quad 2X_3\xrightarrow{\k_3} 2X_2\}~.
\]
This network is two-dimensional, as 
the total amount of $X_1,X_2,X_3$ is conserved. For every vector of positive rate constants $\kappa$, the system $(G,\kappa)$ is nondegenerately multistationary and also has ACR in $X_3$ with ACR-value $\k_2/ (2\k_3)$.
Details are given in the proof of Proposition \ref{prop:MSS_conserved}, below, which pertains to a family of networks that includes the network $G$.
 \end{example}

\subsubsection{Non-full-dimensional networks} \label{sec:cons-law}

The bimolecular network in Example~\ref{example:min3} is the $n=3$ case of the networks $G_n$ that we introduce in the next result.  These networks have the property that every reactant complex is bimolecular, but (when $n\geq 4$) one of product complexes is not.

\begin{proposition}[ACR and multistationarity for all rate constants] \label{prop:MSS_conserved}
For all $n \geq 3$, consider the following network with $n$ species, $n$ reactant complexes, and $n$ reactions:
	\[
	G_n ~=~
	\left\{X_1+X_2 \xrightarrow{\k_1} 2X_3+\sum_{j=4}^nX_j, ~ X_3\xrightarrow{\k_2} X_1 ,  ~2X_3\xrightarrow{\k_3} 2X_2\right\} \bigcup	\left\{
	X_4\overset{\k_{4}}\to 0,~
	\dots 	,~		
		X_n\overset{\k_{n}}\to 0
		\right\}~.
	\]
Each such network $G_n$ satisfies the following:
\begin{enumerate} 
	\item there is a unique (up to scaling) conservation law, which is given by $x_1+x_2+x_3=T$, where $T$ represents the total concentration of species $X_1,X_2,X_3$; and
	\item for every vector of positive rate constants $\kappa \in \mathbb{R}^n_{>0}$, the system $(G_n,\kappa)$ is nondegenerately multistationary and also has ACR in species $X_3$, $X_4, \dots, X_n$.
\end{enumerate}
\end{proposition}

\begin{proof}
Fix $n \geq 3$.  The mass-action ODEs for $G_n$ are as follows:
\begin{align*}
    &\frac{dx_1}{dt} ~=~
    	-\k_1x_1x_2 + \k_2 x_3\\
    &\frac{dx_2}{dt} ~=~
-\k_1x_1x_2 + 2\k_3 x_3^2\\
    &\frac{dx_3}{dt} ~=~
2\k_1x_1x_2 - \k_2x_3 - 2\k_3x_3^2\\
    &\frac{dx_j}{dt} ~=~
\k_1 x_1 x_2 - \k_{j} x_j ~~\quad \text{ for } j\in \{4,\ldots, n\}~.
\end{align*}

The network $G_n$ has exactly one conservation law (up to scaling), and it is given by $x_1+x_2+x_3=T$.   
Additionally, using the first two ODEs, we compute that the value of species $X_3$ at all positive steady states is $\frac{\k_2}{2\k_3}$.  
Next, we use this steady-state value for $X_3$, together with the first and fourth ODEs, to obtain the expression $\frac{\k_2^2}{2\k_3\k_n} $ for the steady-state value for $X_j$, for $j \geq 4$.
Thus, ACR in $X_3,X_4,\dots, X_n$ will follow once we confirm the existence of positive steady states.

Next, we investigate the steady-state values of $X_1$ and $X_2$.  Using the steady-state value of $X_3$, the conservation law, and the first ODE, we see that the steady-state values of $X_1$ and $X_2$ correspond to the intersection points of the line $x_1+x_2 + \frac{\k_2}{2\k_3} = T$ and the curve $x_1 x_2 = \frac{\k_2^2}{2\k_1\k_3}$. This is depicted qualitatively below (by [green] dashed lines and a [red] solid curve, respectively).

	\begin{center}
	\includegraphics[scale=0.35]{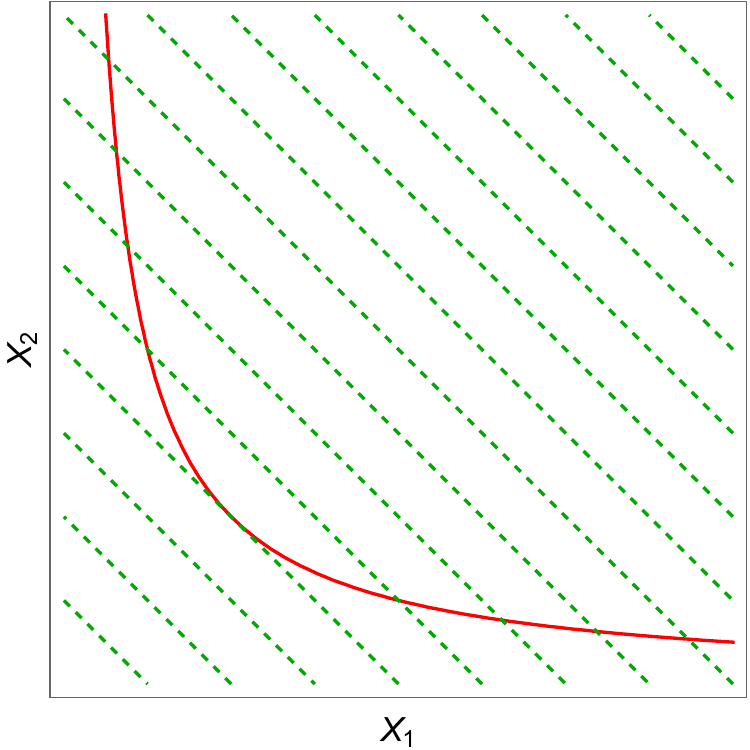}
	\end{center}

It follows that, given any vector of positive rate constants $\kappa \in \mathbb{R}^n_{>0}$, when $T$ is sufficiently large, there are two pairs of (nondegenerate) positive steady-state values for $X_1$ and $X_2$, and so 
$(G_n, \kappa)$ is nondegenerately multistationary (and thus admits a positive steady state, and so has ACR).
\end{proof}

\subsubsection{Full-dimensional networks with $3$ species} \label{sec:full-d-3-species}
Consider a bimolecular network $G$ that has $3$ species.  
We saw that if  $G$ admits ACR and nondegenerate multistationarity simultaneously, then $G$ has at least $3$ reactant complexes
(Theorem~\ref{thm:not_fulldim}).  
If, however, $G$ is full-dimensional, then more reactants are required, as stated in the following result.

\begin{proposition}[Minimum number of reactants for full-dimensional $3$-species networks] \label{prop:fulldim}
Let $G$ be a full-dimensional bimolecular reaction network with exactly 3 species. If there exists a vector of positive rate constants $\kappa^* $ such that $(G,\kappa^*)$ has ACR and is nondegenerately multistationary, then $G$ has at least 5 reactant complexes (and hence at least 5 reactions)
\end{proposition}
\noindent
Proposition~\ref{prop:fulldim} is a direct consequence of 
Propositions~\ref{prop:rank-of-N}(3) and~\ref{prop:2-or-3-rxns3species}, and a stronger version of this result appears 
in the next section (Theorem \ref{theorem:smallestACRMSS}). 
Proposition~\ref{prop:fulldim} implies that if a full-dimensional bimolecular network with 3 species and fewer than 5 reactions has both ACR and multistationarity, then this coexistence happens in a degenerate way. 
We illustrate this situation with two examples, and then characterize all such networks with exactly 4 reactant complexes (Proposition~\ref{prop:4rxns3species}).

\begin{example} \label{ex:3-species-4-reactions-degenerate}
Consider the following full-dimensional network with $3$ species, $4$ reactions, and $4$ reactant complexes: 
$\{2Z \to Z,~ X+Y \to Z \to Y+Z,~ 0 \to X\}$.  
When all rate constants are $1$, the mass-action ODEs are as follows:
\begin{align*}
	\frac{dx}{dt} \quad &= \quad 1-xy\\
	\frac{dy}{dt} \quad &= \quad z - xy\\
	\frac{dz}{dt} \quad &= \quad -z^2 + xy~.
\end{align*}
For this system, 
the set of positive steady states is $\{ (x,y,z) \in \mathbb{R}^3_{>0} \mid xy=z=1\}$, 
and every positive steady state is degenerate.
We conclude that this system is multistationary (but degenerately so) and has 
ACR in species $Z$ (with ACR-value $1$).
\end{example}

\begin{example} \label{ex:3-species-4-reactions-degenerate-set2}
	Consider the following  network: 
	$\{ X+Z\to Z,~  Y+Z \leftrightarrows Y\to 0,~ 2X\leftarrow X\to X+Y \}$.  
	Like the network in Example~\ref{ex:3-species-4-reactions-degenerate}, 
	this network is full-dimensional and has 
	 $3$ species, 
	$4$ reactions, and $4$ reactant complexes; however, the set of reactant complexes differs. 
	When all rate constants are $1$, the mass-action ODEs are as follows:
	\begin{align*}
	\frac{dx}{dt} \quad &~=~ \quad x-xz\\
	\frac{dy}{dt} \quad &~=~ \quad x - y\\
	\frac{dz}{dt} \quad &~=~ \quad y - yz~.
	\end{align*}
For this system, 
the set of positive steady states is $\{ (x,y,z) \in \mathbb{R}^3_{>0} \mid x=y,~z=1\}$, 
and every positive steady state is degenerate.
Thus, this system is (degenerately) multistationary and has 
ACR in species $Z$ (with ACR-value $1$).
\end{example}

The next result shows that Examples~\ref{ex:3-species-4-reactions-degenerate} and~\ref{ex:3-species-4-reactions-degenerate-set2} cover all cases of three-species, four-reactant networks with ACR and (degenerate) ACR occurring together, in the sense that these two networks represent the only two possibilities for the set of reactant complexes 
(when a certain full-rank condition is met, which we discuss below in Remark~\ref{rem:rk-condition-in-thm}).

\begin{proposition}[Networks with $3$ species and $4$ reactants]\label{prop:4rxns3species}
	Let $G$ be a full-dimensional bimolecular
	reaction network with exactly 3 species -- which we call $X,Y,Z$ -- and exactly 4 reactant complexes. 
	If $\kappa^* $ is a vector of positive rate constants such that: 
	\begin{itemize}
	\item[(a)] $\operatorname{rank}(N) = 3$, 
	where $N$ is the matrix for $(G,\kappa^*)$ as in~\eqref{eqn:ODEexpression},
	\item[(b)] $(G, \kappa^*)$ has ACR in species $Z$, and
	\item[(c)] $(G,\kappa^*)$ is multistationary (which is degenerately so, by Proposition~\ref{prop:fulldim}),
	\end{itemize}
	then the set of reactant complexes of $G$ is either 
		$\{X,~X+Z,~Y,~Y+Z\}$ or
		$\{0,~X+Y,~Z,~2Z\}$.
\end{proposition}

\begin{proof} 
Let $G$, $\kappa^*$, and $N$ be as in the statement of the proposition.  
In particular, 
$G$ has $3$ species and $4$ reactants, and 
$(G,\kappa^*)$ admits a positive steady state, which we denote by
$(x^*,y^*, \alpha)$ (so $\alpha$ is the ACR-value of $Z$).
Also, $N$ has rank $3$ and so Proposition~\ref{prop:rank-of-N}(2) and its proof imply
that steady-state equations can be ``row-reduced'' so that
 the positive steady states of $(G,\kappa^*)$ are the roots of $3$ binomial equations of the following form: 
	\begin{align*}
	h_1 ~&:=~ { m}_1 - \beta_1 { m}_4 ~=~0 \\
	h_2 ~&:=~ { m}_2 - \beta_2 { m}_4 ~=~0 \\
	h_3~&:=~ { m}_3 - \beta_3 { m}_4 ~=~0 ~,
	\end{align*}
	where
	$\beta_j \in \mathbb{R}$ (for $j=1,2,3$) and 
	 $m_i=x^{a_i} y^{b_i} z^{c_i} $ (for $i=1,2,3,4$) are 4 distinct monic monomials given by the reactant complexes. 
	 Also, each $m_i$ (for $i=1,2,3,4$) has degree at most 2
	 in $x,y,z$ (as $G$ is bimolecular). 
	 In other words, $a_i,b_i,c_i$ are non-negative integers that satisfy the following:
	% INEQUALITY
	 \begin{align} \label{eq:inequ}
	 a_i+b_i+c_i ~\leq~ 2~.
	 \end{align}
%	 (that is, $a_i+b_i+d_i \leq 2$)}. 
%
	We infer that $\beta_1, \beta_2, \beta_3 >0$, because otherwise $h_1=h_2=h_3=0$ would have no positive roots. 
	
	For $i\in\{1,2,3\}$, consider the following, where we recall that  $\alpha$ is the ACR-value of $Z$:
	\begin{align*}
	g_i ~:=~
	h_i|_{z = \alpha}
	~=~
	 d_i x^{a_i} y^{b_i} - d_i' x^{a_4} y^{b_4}~,
	\end{align*} 
	where 
	$d_i:= \alpha^{c_i}>0$ and $d_i':= \beta_i  \alpha^{c_4}>0$. 
	For $i\in\{1,2,3\}$, 
	by construction, $g_i(x^*,y^*)=0$ and so the subset of the positive quadrant $\mathbb{R}^2_{>0}$ defined by $g_i=0$, which we denote by $S_i$, is nonempty.  There are four possible ``shapes'' for each set $S_i$:
	\begin{enumerate}
		\item $S_i = \mathbb{R}^2_{>0}$, when $(a_i,b_i)=(a_4,b_4)$  (and necessarily, $d_i = d_i'$, to avoid $S_i = \emptyset$).
		\item $S_i$ is the horizontal line $y=y^*$, when $a_i = a_4$ and $b_i \neq b_4$.
		\item $S_i$ is the vertical line $x=x^*$, when $a_i \neq a_4$ and $b_i = b_4$.
		\item $S_i$ is a strictly increasing curve (passing through $(x^*,y^*)$) defined by the following equation, when $a_i \neq a_4$ and $b_i \neq b_4$:
		\begin{align*}
		y ~=~ \left( \frac{d_i}{d'_i} \right) ^ {\frac{1}{b_4 - b_i } } x ^{\frac{a_i - a_4 }{b_4 - b_i } } ~. 
		\end{align*}	
	\end{enumerate}
	Any two lines/curves of the form $(2)$--$(4)$ either coincide or intersect only at $(x^*,y^*)$. Hence, 
	the intersection $S_1 \cap S_2 \cap S_3$ is either 
	\begin{enumerate}[(a)]
	\item 
	the single point $(x^*,y^*)$, 
	\item 
	a single line or curve of the form  $(2)-(4)$, or 
	\item 
	the positive quadrant $ \mathbb{R}^2_{>0}$.  
	\end{enumerate}
	By construction and the fact that $\alpha$ is the ACR-value, the set of all positive steady states of $(G,\kappa^*)$ is the set $\{(x,y, \alpha ) \mid (x,y) \in S_1 \cap S_2 \cap S_3\}$.  Hence, in the case of (a), $(G,\kappa^*)$ is not multistationary, which is a contradiction. 

Next, we show that case (c) does not occur. On the contrary, assume that it does. Then $S_1=S_2=S_3=\RR_{> 0}^2$, which implies that $(a_1,b_1)=(a_2,b_2)=(a_3,b_3) = (a_4,b_4)$. 
Since $m_1,m_2,m_3,m_4$ are $4$ distinct monomials, 
it must be that $c_1,c_2,c_3,c_4$ are $4$ distinct non-negative integers.
However, as noted earlier, $c_i\in\{0,1,2\}$ for each $i$, which yields a contradiction.

Finally, we consider case (b).  
This case happens only when one of the following subcases occur:

{\bf Subcase 1:}
			{\em Exactly 1 of the 3 subsets $S_i$ is the positive quadrant, and the other two coincide.} 
			Without loss of generality, assume $S_1 = \mathbb{R}^2_{>0}$ and so $S_2 = S_3 \neq \mathbb{R}^2_{>0}$. 
			Hence, 
			$(a_1,b_1)=(a_4,b_4) \neq
			(a_2,b_2) = (a_3,b_3)$. 
			However, $m_1 \neq m_4$ and $m_2 \neq m_3$, and so:
			\begin{align} \label{eq:inequality-z-coeff}
			c_1~\neq~ c_4 \quad {\rm and} \quad c_2~\neq~c_3~.
			\end{align}
			We rewrite the inequalities~\eqref{eq:inequ}, 
			using 
			the equalities 
			$(a_1,b_1)=(a_4,b_4)$ and 
			$(a_2,b_2) = (a_3,b_3)$:
				 \begin{align} \label{eq:inequ-case-1}
	 		a_1+b_1+c_1 ~\leq~ 2~,
			\quad  
	 		a_1+b_1+c_4 ~\leq~ 2~,
			\quad  
	 		a_2+b_2+c_2 ~\leq~ 2~,
			\quad  
	 		a_2+b_2+c_3 ~\leq~ 2~.
	 		\end{align}
			Finally, Lemma~\ref{lem:involve-all-species} implies that each of species $X$ and $Y$ takes part in some reactant complex, so we obtain the following (again using $(a_1,b_1)=(a_4,b_4)$ and 
			$(a_2,b_2) = (a_3,b_3)$):
			\begin{align} \label{eq:sum}
			a_1 + a_2 ~\geq~ 1
				 \quad {\rm and} \quad 
			b_1 + b_2 ~\geq~ 1~. 
			\end{align}
			The only non-negative solutions to the conditions in~\eqref{eq:inequality-z-coeff}, \eqref{eq:inequ-case-1}, and \eqref{eq:sum} are as follows:
			\begin{itemize}
				\item $a_1=a_3=1$, $a_2=a_3=0$, $b_1=b_4=0$, $b_2=b_3=1$, $\{c_1,c_4\} = \{c_2,c_3\} = \{0,1\}$; 
				\item $a_1=a_3=0$, $a_2=a_3=1$, $b_1=b_4=1$, $b_2=b_3=0$, $\{c_1,c_4\} = \{c_2,c_3\} = \{0,1\}$.
			\end{itemize}
			In all of these solutions, the set of reactant complexes is $\{X, ~X+Z, ~Y, ~Y+Z\}$. %, which completes this case.

{\bf Subcase 2:}
			 {\em Exactly 2 of the 3 subsets $S_i$ are the positive quadrant.} 
			 Without loss of generality, assume that 
			  $S_1 = S_2 = \mathbb{R}^2_{>0} \neq  S_3$.  This implies the following:
			  \begin{align} \label{eq:subcase-2-of-3} 
			  (a_1,b_1)~=~(a_2,b_2)~=~(a_4,b_4) ~\neq~ (a_3,b_3) ~.
			  \end{align}
			  However, $m_1,m_2,m_4$ are $3$ distinct monomials, so $c_1, c_2,c_4$ are
			  $3$ distinct non-negative integers.  Now inequality~\eqref{eq:inequ} implies that 
			  $\{c_1, c_2,c_4\} = \{0,1,2\}$.  Let $i^* \in \{1,2,4\}$ be such that $c_{i^*}=2$.  
			  Next, 
			  the equalities in~\eqref{eq:subcase-2-of-3} and the 
			  inequality~\eqref{eq:inequ} for $i=i^*$ together imply that 
			 $(a_1,b_1)=(a_2,b_2)=(a_4,b_4)=(0,0)$. 
			 Therefore, the set of reactant complexes corresponding to $m_1,m_2,m_4$ is
			 %This gives us the complexes 
			 $\{0, ~Z, ~2Z\}$. 
			 Finally, Lemma~\ref{lem:involve-all-species} implies that the fourth reactant complex must involve both $X$ and $Y$ and so (by bimolecularity) is $X+Y$.  
			 Therefore, the set of reactant complexes is $\{0,~ X+Y, ~Z, ~2Z\}$.

{\bf Subcase 3:}
			{\em None of the subsets $S_i$ are positive quadrants, and the $3$ sets coincide.}
			This implies that $(a_1,b_1)=(a_2,b_2)=(a_3,b_3)  \neq (a_4,b_4)$. 
			These conditions are symmetric to those in subcase~2, and so the reactant complexes are $\{0,~ Z, ~2Z, ~X+Y\}.$  This completes subcase 3 (and case (b)).
\end{proof}

\begin{remark}[Rank condition] \label{rem:rk-condition-in-thm} 
Proposition~\ref{prop:4rxns3species} includes the hypothesis that the matrix $N$ for the system $(G,\kappa^*)$ has (full) rank $3$. 
If we remove this hypothesis, we can obtain more full-dimensional networks with $3$ species and $4$ reactants that allow ACR and (degenerate) multistationarity to occur together.
We present one such network in Example~\ref{eg:3species-2rank}.
\end{remark}

\begin{example}\label{eg:3species-2rank}
	Consider the following full-dimensional network with 3 species and 4 reactants: 
	\[G:=\{0\to X\to Y \to 2 Y,~~~~ Y\leftarrow Y+Z \to 2 Z\}~. \]
	The system $(G,\kappa^*)$ obtained by setting all the reaction rates to 1 has the following ODEs:
		\begin{align*}
\begin{bmatrix}
dx/dt\\
dy/dt\\
dz/dt
%\frac{dx}{dt}\\
%\frac{dy}{dt}\\
%\frac{dz}{dt}
\end{bmatrix} ~=~
	\begin{bmatrix}
	-1& 0 & 0 & 1 \\
1	& 1 & -1 & 0 \\
	0& 0 & 0 & 0 \\
	\end{bmatrix}
	\begin{bmatrix}
	x\\
	y\\
	yz\\
	1
	\end{bmatrix}	
	~=~ N  
	\begin{bmatrix}
	x\\
	y\\
	yz\\
	1
	\end{bmatrix}	~.
	\end{align*}
	The matrix $N$ (defined above) has rank $2$, the set of positive steady states is $\{ (x,y,z) \in \mathbb{R}^3_{>0} \mid x=1,~ y(z-1)=1\}$, and every positive steady state is degenerate.  
	Thus, this system is (degenerately) multistationary and has 
	ACR in species $X$ (with ACR-value $1$).
\end{example}

In the next section, we see that the exceptional networks in Proposition~\ref{prop:4rxns3species} -- namely, full-dimensional, three-species networks with reactant-complex set $ \{0,Z,2Z,X+Y\}$ or $\{X,Y,X+Z,Y+Z\}$ -- do not have unconditional ACR.  Indeed, this fact is a direct consequence of a more general result concerning networks with $n$ species and $n+1$ reactants (Theorem~\ref{thm:generaln1reactants}).

\section{Main results on general networks} \label{sec:results-general}
The results in the prior section pertain to networks that are bimolecular, while here we analyze networks that need not be bimolecular.
We consider full-dimensional networks (Section~\ref{sec:full-dim}) and non-full-dimensional networks (Section~\ref{sec:not-full-dim}) separately.

\subsection{Full-dimensional networks} \label{sec:full-dim}

In Proposition~\ref{prop:MSS_conserved}, we saw a family of networks 
that admit
 ACR and nondegenerate multistationarity together. 
 These networks 
have $n$ reactants (where $n$ is the number of species), but are not full-dimensional.  
In this subsection, we show that for full-dimensional networks, 
the coexistence of 
ACR and nondegenerate multistationarity requires at least $n+2$ reactants (Theorem~\ref{theorem:smallestACRMSS}). 
We also show that this lower bound is tight (Proposition~\ref{prop:family-min-num-reactants}).  
Additionally, we consider full-dimensional networks with only $n+1$ reactants and show that if such a network is multistationary (even if only degenerately so), then the network can not have unconditional ACR
(Theorem~\ref{thm:generaln1reactants}).

\begin{theorem}[Minimum number of reactants for full-dimensional networks] \label{theorem:smallestACRMSS}
	Let $G$ be a full-dimensional reaction network with $n$ 
	species.
	If there exists a vector of positive rate constants $\kappa^*$ such that the mass-action system $(G,\kappa^*)$ has ACR and also is nondegenerately multistationary, then 
	$n \geq 2$ and
	 $G$ has at least $n+2$ reactant complexes and hence, at least $n+2$ reactions.
\end{theorem}

\begin{proof}
It follows readily from definitions that ACR and multistationarity do not coexist in networks with only one species, so $n \geq 2$.
We proceed by contrapositive. 
We consider two cases.  
If $G$ has at most $n$ reactant complexes, then 
Proposition~\ref{prop:2-or-3-rxns3species} (which requires $n \geq 2$) implies that every positive steady state of $(G,\kappa^*)$ is degenerate and so $(G,\kappa^*)$ is {\em not} nondegenerately multistationary.
In the remaining case, when $G$ has $n+1$ reactant complexes, Proposition~\ref{prop:rank-of-N}(3)
implies that $(G,\kappa^*)$ is {\em not} nondegenerately multistationary.
\end{proof}

The next example shows that the bound in Theorem~\ref{theorem:smallestACRMSS} is tight for $n=2$.

\begin{example}
\label{eg:2speciesminreactions} 
The following network is full-dimensional and has $2$ species, $4$ reactant complexes, and $4$ reactions (the out-of-order labeling of the rate constants is to be consistent with Proposition~\ref{prop:family-min-num-reactants}, which appears later):		
		\[
		\{
		A+B \overset{\kappa_1 } \to 2B \overset{\kappa_3 } \to 2B+A
		~,~
		B \overset{\kappa_2 } \to 0 \overset{\kappa_4 } \to A		 			  		 			  
		\}~.
		\]
Observe that all reactant complexes are bimolecular, but one of the product complexes is not. 
In the next result, we show that this network exhibits ACR (in species $A$ with ACR-value $\kappa_2/\kappa_1$) and nondegenerate multistationarity when $\kappa_2^2 > 4 \kappa_3 \kappa_4.$ (Proposition~\ref{prop:family-min-num-reactants}).  
Among full-dimensional networks for which ACR and nondegenerate multistationarity coexist,
this network is optimal 
in the sense that it has the fewest possible
species,  reactant complexes, and reactions (by 
Theorem~\ref{theorem:smallestACRMSS}).
\end{example}

In the next result, we generalize the network in Example~\ref{eg:2speciesminreactions} to a family of networks that show that the lower 
bound on the number of reactions in Theorem~\ref{theorem:smallestACRMSS} is tight for all $n$.
The networks in the following result are also optimal in terms of the molecularity of the reactant complexes (they are bimolecular), although two of the product complexes have high molecularity.

\begin{proposition}\label{prop:family-min-num-reactants}
For all $n \geq 2$, consider the following full-dimensional network with 
$n$ species, $n+2$ reactions, and $n+2$ reactant complexes:
	\begin{align*} %\label{eq:network-family}
	G_n ~&=~
	\left\{
	X_1+X_2 \overset{\kappa_1 } \to 2X_2 +
		X_3 + \dots + X_n
		%\sum_{j=3}^{n}X_j
	,~
	X_2 \overset{\kappa_2 } \to 0,~
	2X_2	\overset{\kappa_3 } \to 2X_2+X_1
	,~
	0 \overset{\kappa_4 } \to X_1		 			  
	\right\}
	\\
	\notag
	& \quad \quad \quad \bigcup 
	\{ X_j\overset{\kappa_{j+2}}\to 0 \mid 3 \leq j \leq n\}.
	\end{align*}
For every vector of positive rate constants $\kappa^*$ for which $(\kappa^*_2)^2 > 4 \kappa^*_3 \kappa^*_4 $, the system $(G_n, \kappa^*)$ has nondegenerate multistationarity and has ACR in species $X_1$.
\end{proposition}

\begin{proof}
The mass-action ODEs are given by:
	\begin{align*}
{\frac{dx_1}{dt}}	&~=~\kappa_3 x_2^2 - \kappa_1 x_1 x_2 + \kappa_4\\
{\frac{dx_2}{dt}}	&~=~\kappa_1 x_1 x_2-\kappa_2x_2\\
{\frac{dx_j}{dt}}	&~=~\kappa_1 x_1 x_2 - \kappa_{j+2} x_j ~~\quad \text{ for } j\in {3,\ldots, n.}
\end{align*}
The steady-state equation for $X_2$ implies that $x_1=\kappa_2/\kappa_1$ at all positive steady states, so there is ACR in $X_1$ (whenever positive steady states exist).  Next, the steady-state equations for $X_1$ and $X_2$ imply that the steady state values of $X_2$ are 
$x_2^{\pm} :=
\frac{k_2 \pm \sqrt{k_2^2 - 4k_3k_4}}{2 k_3}$.  Both of these steady state values are positive precisely when the discriminant $k_2^2 - 4k_3k_4$ is positive (this is a straightforward computation; alternatively, see~\cite[Proposition~2.3]{dennis-shiu}).  Now we use the steady-state equation for $X_j$, where $j\geq 3$, to compute the two positive steady states that exist whenever $(\kappa^*_2)^2 > 4 \kappa^*_3 \kappa^*_4 $:
\begin{align*} \label{eq:2-steady-states}
	& 
	\left(
	x_1^*,~
	x_2^{+},~
	\frac{\kappa_1}{\kappa_3} x_1^*x_2^{+},~
	\dots,
	~
	\frac{\kappa_1}{\kappa_n} x_1^*x_2^{+}
	\right)
	%\\
	%& 
	\quad {\rm and } \quad 
	\left(
	x_1^*,~
	x_2^{-},~
	\frac{\kappa_1}{\kappa_3} x_1^*x_2^{-},~
	\dots,
	~
	\frac{\kappa_1}{\kappa_n} x_1^*x_2^{-}
	\right)	~,
\end{align*}
where $x^*_1:=\kappa_2/\kappa_1$.
Finally, nondegeneracy can be checked by a computing the Jacobian matrix.
\end{proof}

Our next result concerns full-dimensional networks with one more reactant than the number of species, as follows.  

\begin{theorem}[Networks with $n+1$ reactants]\label{thm:generaln1reactants}
	Let $G$ be a full-dimensional network, with $n$ species and exactly $n+1$ reactant complexes.  If 
	$G$ is multistationary, then there exists a vector of positive rate constants $\kappa$ such that $(G,\kappa)$ has \uline{no} positive steady states, and hence $G$ does \uline{not} have unconditional ACR.
\end{theorem}

\begin{proof}
Assume that $G$ is full-dimensional, has exactly $n+1$ reactant complexes (where $n$ is the number of species), and is multistationarity.  By definition, there exists $\kappa^* \in \mathbb{R}^r_{>0}$, where $r$ is the number of reactions, such that $(G,\kappa^*)$ is multistationary.  Let $N$ be the $n \times (n+1)$ matrix arising from $(G,\kappa^*)$, as in~\eqref{eqn:ODEexpression}; and let $A$ be the $n \times n$ matrix defined by $G$, as in Proposition~\ref{prop:fulldim_noPSS}.  

We claim that  $\operatorname{rank}(N) \leq n-1$ or  $\operatorname{rank}(A) \leq n-1$.  Indeed, if $\operatorname{rank}(N) = n$ and  $\operatorname{rank}(A) = n$, then the proof of Proposition~\ref{prop:fulldim_noPSS} shows that $(G,\kappa^*)$ is not multistationary, which is a contradiction.

If  $\operatorname{rank}(N) \leq n-1$, then Proposition~\ref{prop:fulldim-fullrank-PSS}(2) implies that 
there exists
$\kappa^{**} \in \mathbb{R}^r_{>0}$ such that $(G,\kappa^{**})$ has {no} positive steady states.  
Similarly, in the remaining case, when $\operatorname{rank}(N) =n$ and $\operatorname{rank}(A) \leq n-1$, 
the desired result follows directly from
Proposition~\ref{prop:fulldim_noPSS}(2).
\end{proof}

\subsection{Non-full-dimensional networks} \label{sec:not-full-dim}

In an earlier section, we saw a family of networks 
with $n$ species, $n$ reactant complexes, and exactly one conservation law, for which 
ACR and nondegenerate multistationarity coexist (Proposition \ref{prop:MSS_conserved}).  Our next result shows that this $n$
is the minimum number of reactant complexes (when there is one conservation law), and, furthermore, as the number of conservation laws increases, the minimum number of reactant complexes required decreases.

\begin{theorem}[Minimum number of reactants] \label{thm:min-num-reac-with-k-cons-law}
	Let $G$ be a reaction network with $n \geq 3$ species and $k \geq 1$ conservation laws (more precisely, $G$ has dimension $n-k$).  If there exists a vector of positive rate constants 
		$\kappa^*$  
		such that the system $(G,\kappa^*)$ is nondegenerately multistationary and has ACR in some species, 
	then $G$ has at least $n-k+1$ reactant complexes.
\end{theorem}

\begin{proof} If $G$ has $k \geq 1$ conservation laws and at most $n-k$ reactant complexes, then Propositions~\ref{prop:2-or-3-rxns3species}(1) and
~\ref{prop:min-num-reac-with-k-cons-law} together imply that $G$ is {\em not} nondegenerate multistationarity.
\end{proof}

As noted earlier, the bound in Theorem~\ref{thm:min-num-reac-with-k-cons-law} 
is tight for $k=1$, due to Proposition~\ref{prop:MSS_conserved}.  
We also know that, for $k=n-1$, the bound holds vacuously (Proposition~\ref{prop:1-d-no-coexistence}).  Our next result shows that the bound is also tight for all remaining values of $k$ (namely, $ 2 \leq k \leq n-2$).

\begin{proposition}\label{conj:boundTight}
Let $n \geq 3$, and let $k \in \{2,3,\dots, n-2\}$.  If $k \neq n-2$, consider the following network:
\begin{align*}
	G_{n,k} ~=~
	&\left\{X_1+X_2+\sum_{j=n+2-k}^{n}X_j \xrightarrow{\k_1} 2X_3+\sum_{j=4}^nX_j, ~ X_3\xrightarrow{\k_2} X_1 ,  ~2X_3\xrightarrow{\k_3} 2X_2\right\}\\
 &\bigcup	\left\{
	X_{4}\overset{\k_{4}}\to 0,~
	\dots 	,~		
		X_{n+1-k}\overset{\k_{n-k+1}}\to 0
		\right\}~.
	\end{align*}
On the other hand, if $k = n-2$, consider the following network:
\begin{align*}
	G_{n,k} ~=~
	\left\{X_1+X_2+\sum_{j=4}^{n}X_j \xrightarrow{\k_1} 2X_3+\sum_{j=4}^nX_j, ~ X_3\xrightarrow{\k_2} X_1 ,  ~2X_3\xrightarrow{\k_3} 2X_2\right\}
	\end{align*}
Each such network $G_{n,k}$ satisfies the following:
\begin{enumerate} 
	\item $G_{n,k}$ has $n$ species, $n-k+1$ reactants, and $n-k+1$ reactions;
	\item $G_{n,k}$ has dimension $n-k$, and the following are $k$ linearly independent conservation laws: $x_1+x_2+x_3=T$ and $x_j = T_j$ for $j\in\{n-k+2,\dots,n\}$.
	\item for every vector of positive rate constants $\kappa$, the system $(G_n,\kappa)$ is nondegenerately multistationary and also has ACR in species $X_3$, $X_4, \dots, X_{n-k+1}$. 
\end{enumerate}
\end{proposition}

\begin{proof}
This result can be checked directly, in a manner similar to the proof of 
Proposition \ref{prop:MSS_conserved}.
Indeed, for every vector of positive rate constants, there is ACR in species $X_3,X_4, \ldots,X_{n-k+1}$ and exactly two nondegenerate positive steady states when $T$ is large enough.
\end{proof}

The reaction networks in Proposition \ref{conj:boundTight} are not bimolecular, and they contain reactions with many catalyst-only species (namely $X_{n-k+2},\dots,X_n$). 
We do not know whether there exist reaction networks 
that 
are bimolecular and do not contain reactions with catalyst-only species, and yet (like the networks in Proposition \ref{conj:boundTight}) 
show that the lower
bound in Theorem~\ref{thm:min-num-reac-with-k-cons-law} is tight.

\section{Discussion} \label{sec:discussion}
In this article, we proved lower bounds in terms of the dimension and the numbers of species, reactant complexes (and thus reactions), and all complexes (both reactant and product complexes) needed for the coexistence of ACR and nondegenerate multistationarity.  Additionally, we 
showed that these bounds are tight, via the network 
$\{A+B \to 2C \to 2B,~ C\to A \}$ (Example~\ref{example:min3}).

Networks like the one in Example \ref{example:min3} contain special structures that may be biologically significant. 
Exploring such structures will aid in establishing design principles for creating networks with ACR and multistationarity. We plan to explore such networks and their architecture in the future. 

In the present work, our interest in multistationarity comes from the fact that it is a necessary condition for multistability. Another interesting direction, therefore, is to investigate conditions for coexistence of ACR and {\em multistability}, rather than multistationarity. The ``minimal" networks in the current work admit only two positive steady states and are not multistable.  Hence, we conjecture that the lower bounds (on dimension and the numbers of species, reactant complexes, and all complexes) for the coexistence of ACR and multistability are strictly larger than the bounds proven here for ACR and multistationarity. 

Finally, we are interested in the conditions for the coexistence of other combinations of biologically significant dynamical properties, such as ACR and oscillations. In addition to the minimum requirements for their coexistence, we also hope to discover new network architectures or motifs that can be used to design synthetic networks possessing these dynamical properties.

%*******************************************************************
%Acknowledgements
%*******************************************************************
\subsection*{Acknowledgements}
{\small
This project began at an AIM workshop on ``Limits and control of stochastic reaction networks'' held online in July 2021.  
AS was supported by the NSF (DMS-1752672).  The authors thank Elisenda Feliu, Oskar Henriksson, Badal Joshi, and Beatriz Pascual-Escudero for many helpful discussions.}

\bibliographystyle{plain}

\bigskip \bigskip

\noindent
\footnotesize {\bf Authors' addresses:}

\smallskip

\noindent Nidhi Kaihnsa,
\  University of Copenhagen \hfill  {\tt nidhi@math.ku.dk}

\noindent Tung Nguyen, Texas A\&M University
\hfill {\tt daotung.nguyen@tamu.edu}

\noindent Anne Shiu, Texas A\&M University
\hfill {\tt annejls@tamu.edu}

\end{document}